\documentclass[a4paper,12pt]{scrartcl}

\usepackage[backref=true,style=alphabetic]{biblatex}
\bibliography{dynsys}

\usepackage[left=2.5cm,right=2.5cm,top=2.7cm,bottom=2.7cm]{geometry} 
\usepackage[utf8]{inputenc}
\usepackage[english]{babel}
\usepackage{amssymb}
\usepackage{amsmath}
\usepackage{latexsym}
\usepackage[bookmarks]{hyperref}
\usepackage{amsthm}
\usepackage{graphicx}
\usepackage{bbm}
\usepackage{verbatim}
\usepackage{epigraph}
\usepackage{mathtools}
\usepackage{authblk}
\usepackage[pdflatex]{crop}
\usepackage{tikz}
\usepackage{tikz-cd}
\usepackage{wasysym}
\usepackage{adjustbox}
\usepackage{microtype}
\usepackage{chngcntr}
\usepackage{csquotes}
\usepackage[capitalize]{cleveref}
\usepackage{stmaryrd}	
\usepackage{mathrsfs}	
\allowdisplaybreaks
\usetikzlibrary{shapes.geometric,fit}

\setkomafont{disposition}{\normalfont\bfseries}
\newcommand{\auth}{Tobias Fritz, Paolo Perrone, Sharwin Rezagholi}
\newcommand{\tit}{Probability, valuations, hyperspace: Three monads on Top and the support as a morphism}
\newcommand{\kw}{Categorical probability, monads, continuous valuations, $\tau$-additive probability measures, hyperspace, powerdomains}

\hypersetup{
pdfauthor={\auth},%
pdftitle={\tit},%
colorlinks, linktocpage=true, pdfstartpage=1, pdfstartview=FitV,%
breaklinks=true, pdfpagemode=UseNone, pageanchor=true, pdfpagemode=UseOutlines,%
plainpages=false, bookmarksnumbered, bookmarksopen=true, bookmarksopenlevel=1,%
hypertexnames=true, pdfhighlight=/O,%
urlcolor=black, linkcolor=black, citecolor=black, 
}
\numberwithin{equation}{section}

\theoremstyle{plain}
\newtheorem{thm}{Theorem}[section]
\newtheorem{lemma}[thm]{Lemma}
\newtheorem{prop}[thm]{Proposition}
\newtheorem{cor}[thm]{Corollary}

\newtheorem{deph}[thm]{Definition}
\newtheorem{nota}[thm]{Notation}

\theoremstyle{definition}
\newtheorem{remark}[thm]{Remark}
\newtheorem{eg}[thm]{Example}
\Crefname{equation}{}{}		

\newcommand{\N}{\mathbb{N}}

\newcommand{\R}{\mathbb{R}}

\newcommand{\Rplusext}{[0,\infty]}	
\newcommand{\sgn}{\mathrm{sgn}}		
\newcommand{\ngs}{\mathrm{ngs}}		

\newcommand{\cat}[1]{{\mathsf{#1}}} 
\newcommand{\ar}[2][]{\arrow{#2}{#1}}
\newcommand{\idar}[2][]{\arrow[equal]{#2}{#1}} 
\DeclareMathOperator{\1}{\mathbbm{1}}

\newcommand{\id}{\mathrm{id}} 

\newcommand{\pair}[2]{\left\langle  {#1}  ,\hspace{.5pt}  {#2}  \right\rangle} 

\newcommand{\join}{\vee}
\DeclareMathOperator*{\bigjoin}{\bigvee}
\newcommand{\intersection}{\cap}
\DeclareMathOperator*{\bigintersection}{\bigcap}
\newcommand{\union}{\cup}

\DeclareMathOperator{\cl}{\mathrm{cl}}
\newcommand{\supp}{\mathrm{supp}}

\DeclareMathOperator{\Hit}{Hit}

\newcommand{\U}{\mathcal{U}}
\newcommand{\E}{\mathcal{E}}
\newcommand{\e}{\mathrm{E}}
\DeclareMathOperator{\eval}{\varepsilon}

\DeclareMathOperator{\op}{\mathcal{O}}

\let\originalleft\left
\let\originalright\right
\renewcommand{\left}{\mathopen{}\mathclose\bgroup\originalleft}
\renewcommand{\right}{\aftergroup\egroup\originalright}

\usepackage{enumitem}
\setlist[enumerate, 1]{label=(\alph{enumi}), ref=(\alph{enumi})}
\setlist[enumerate, 2]{label=(\roman{enumii}), ref=(\roman{enumii})}

\tikzset{%
    bullet/.style={
       fill=black,
       circle,
       minimum width=1pt,
       inner sep=1pt
     },
     relation/.style={
       -,
       thick,
       shorten <=2pt,
       shorten >=2pt
     },
     function/.style={
       ->,
       thick,
       shorten <=2pt,
       shorten >=2pt
     },
     every fit/.style={
       ellipse,
       draw,
       inner sep=0pt
     }
}

\title{Probability, valuations, hyperspace:\\ Three monads on Top and the support as a morphism}
\author[1]{Tobias Fritz\footnote{tobias.fritz [at] uibk.ac.at}}
\author[2]{Paolo Perrone\footnote{paolo.perrone [at] cs.ox.ac.uk}}
\author[3,4]{Sharwin Rezagholi\footnote{sharwin.rezagholi [at] emergentec.com}}

\affil[1]{\small Department of Mathematics, University of Innsbruck, Austria}
\affil[2]{\small Department of Computer Science, University of Oxford, United Kingdom}
\affil[3]{\small St. Petersburg School of Physics, Mathematics, and Computer Science, National Research University HSE, Russia}
\affil[4]{\small emergentec biodevelopment, Vienna, Austria}

\date{\today}

\begin{document}

\newgeometry{top=1cm,bottom=1cm,left=1cm, right=1cm} 

\maketitle

\begin{abstract}
We consider three monads on $\cat{Top}$, the category of topological spaces, which formalize topological aspects of probability and possibility in categorical terms. The first one is the Hoare hyperspace monad $H$, which assigns to every space its space of closed subsets equipped with the lower Vietoris topology. The second one is the monad $V$ of continuous valuations, also known as the extended probabilistic powerdomain. We construct both monads in a unified way in terms of double dualization. This reveals a close analogy between them, and allows us to prove that the operation of taking the support of a continuous valuation is a morphism of monads $V \to H$. In particular, this implies that every $H$-algebra (topological complete semilattice) is also a $V$-algebra. We show that $V$ can be restricted to a submonad of $\tau$-smooth probability measures on $\cat{Top}$. By composing these morphisms of monads, we obtain that taking the supports of $\tau$-smooth probability measures is also a morphism of monads.

\end{abstract}
\restoregeometry 

\newpage

\newgeometry{top=2cm,bottom=1cm} 
\tableofcontents
\restoregeometry 


\section{Introduction}

In recent decades, aspects of measure and probability theory have been reformulated in categorical terms using the categorical structure of monads in the sense of Eilenberg and Moore \cite{eilenberg65}. All probability monads are variations on the distribution monad on $\cat{Set}$ (see, for example,~\cite{jacobs11}), whose underlying functor assigns to a set the set of its finitely supported probability distributions, or, equivalently, the set of formal finite convex combinations of its elements.
Close relatives of the distribution monad are used to treat probability measures in the sense of measure theory. These monads live on suitable categories of spaces with analytic structure, for example the category of measurable spaces, compact Hausdorff spaces, or complete metric spaces.
The monad approach has two main features: conditional probabilities, in the sense of Markov kernels, arise as Kleisli morphisms \cite{girymonad}; and it provides a conceptually simple definition of integration or expectation on all algebras of the monad~\cite[Chapter~1]{thesis}.

In this paper, we consider two monads of this type on $\cat{Top}$, the category of topological spaces and continuous maps. Concretely, we develop the monad of continuous valuations $V$, and the monad of $\tau$-smooth Borel probability measures $P$, which is a submonad of $V$. Our treatment of $V$, and partly also our treatment of $P$, is largely a review of known material presented in a systematic fashion. Our exposition shows how to exploit duality theory for continuous valuations to obtain a simple description of $V$ (as particular functionals on lower semicontinuous functions), and how to use the embedding of $P$ as a submonad of $V$ to reason about $P$ in similar terms.

In some situations, one may only be interested in whether an event is possible at all rather than in its numerical likelihood or propensity. Computer scientists call this situation \emph{nondeterminism}. This distinction between possibility and impossibility can be treated via monads which are similar to probability monads. Instead of assigning to a space $X$ the collection of probability measures or valuations of a certain type, one assigns to $X$ the collection of subsets of a certain type, where one can think of a subset as specifying those outcomes which are possible. The simplest such monad is arguably the finite powerset monad on $\cat{Set}$ \cite[Example 4.18]{manes03}, which assigns to every set the collection of its finite subsets. In this paper, we consider a close relative of this monad on $\cat{Top}$, namely the Hoare hyperspace monad $H$. Its underlying functor assigns to every topological space the space of its closed subsets, seen as particular functionals on open subsets, just as valuations can be seen as functionals on lower semicontinuous functions. While this is also mostly known, our systematic exposition is of interest insofar as our treatment of $H$ is perfectly parallel to our treatment of $V$, which suggests that both of these monads are instances of a general construction (still an open question).

It is elementary to verify that the finite distribution monad and the finite powerset monad on $\cat{Set}$ are related by a morphism of monads, namely the natural transformation which assigns to a finitely supported probability measure its support, which is the subset of elements that carry nonzero weight. That this transformation is a morphism of monads comprises the statement that the support of a convex combination of finitely supported probability measures is given by the union of the supports of the contributing measures. The main new result of this paper is that this holds on arbitrary topological spaces: forming the support is a morphism of monads $V \to H$, which maps a continuous valuation on a space $X$ to a closed subset of $X$. Restricting this transformation to the submonad $P$ of $V$ results in a morphism of monads $P \to H$, which maps every $\tau$-smooth probability measure on a topological space to its support. We believe that this is the most general context in which it is meaningful to talk about the supports of (unsigned) measures; since the support, being a topological concept, is not defined in a purely measure-theoretic setting for non-atomic measures.

From the point of view of denotational semantics, these monads model probabilistic and nondeterministic computation. Our morphism $V \to H$ yields a continuous map from a probabilistic powerspace to the possibilistic Hoare powerspace that respects the respective monad structures. This formalizes the passage from probabilistic computation to nondeterministic computation.

In summary, we study the following three monads on $\cat{Top}$, the category of topological spaces and continuous maps: The monad $H$ of closed subsets (\Cref{hyperspace}), the monad $V$ of continuous valuations (\Cref{sec_valuations}), and the monad $P$ of $\tau$-smooth Borel probability measures (\Cref{probability}). The monad $H$ is a generalization of the \emph{Hoare powerdomain}~\cite[Section~6.3]{schalk}.
The monad $V$ is also known as the \emph{extended probabilistic powerdomain}~\cite{AJK}. In contrast to $H$ and $V$, the monad $P$ has, to the best of our knowledge, not been considered before in this generality. 
The first part of this paper (\Cref{hyperspace,sec_valuations}) introduces these monads and mostly contains known results, but our analogous constructions of the monads seems to be novel: we define them through \emph{double dualization}, a common theme in the theory of monads \cite{lucyshyn-wright}. The idea is that measures, as well as closely related objects, can be seen as dual to functions, which are themselves dual to points. From the point of view of functional analysis, this amounts to versions of the well-known Markov--Riesz duality. 
From the point of view of theoretical computer science, this states that these monads behave \emph{like} submonads of a continuation monad. This is not technically true, since $\cat{Top}$ does not have the relevant exponential objects. Indeed our double dualization construction, unlike existing categorical approaches to double dualization monads, does not rely on Cartesian closure. However, it recovers the exponential objects whenever they exist (see \Cref{functop}). This duality theory turns out to be of great utility in the construction of the monad structures and the proofs that the monad axioms are satisfied. 
It moreover makes the structural similarity between the Hoare hyperspace monad $H$ and the valuation monad $V$ more precise, and yields a structurally simple proof of our main result, that the support is a morphism of monads.

Technically, we show that Scott continuous modular maps from a frame of open sets $\op(X)$ to $\{0,1\}$ are in canonical bijection with closed sets in $X$, and that the topology of pointwise convergence for such functionals corresponds to the lower Vietoris topology on the space of closed sets $HX$ (\Cref{closedduality}). In particular, a closed set $C$ assigns the truth value $1$ to an open set $U$ if an only if $C\intersection U$ is nonempty.
The valuation monad $V$ has a similar duality theory: continuous valuations on a space $X$ are in canonical bijection with Scott continuous modular functionals on the space of lower semicontinuous functions on $X$ (\Cref{daniellval}). We define the monad structures on $H$ and $V$ in terms of these dualities.

In \Cref{probability}, we show that the functor of $\tau$-smooth Borel probability measures is a submonad of $V$. This seems to be the most general probability monad on topological spaces appearing in the literature, as it is defined on the entire category $\cat{Top}$. Its restriction to the subcategory of compact Hausdorff spaces is the Radon monad~\cite{swirszcz,keimel}.

In \Cref{secmonadmorph}, we define the support of a continuous valuation, and prove that the operation of taking the support is a morphism of monads from $V$ to $H$. This operation can be described in the following way. Given a valuation $\nu$, its support is the unique closed set $\supp(\nu)$ such that, for each open set $U$, the set $\supp(\nu)$ intersects $U$ if and only if $\nu(U)$ is strictly positive. From the possibilistic point of view, an open set $U$ is possible if an only if it has positive probability. In this way, the support induces a map $\supp: VX \to HX$ from valuations to closed subsets. We prove that this map is continuous, natural, and compatible with the two monad structures.

A feature that closed sets, valuations, and measures share is the possibility of forming products and marginals (or projections). This is encoded in the fact that the monads in question are commutative, or, equivalently, symmetric monoidal (see \Cref{comm_monads}). These standard formal constructions yield the familiar notions of products and projections of closed sets, and of product and marginal probability measures, and we discuss them for each monad at the end of the respective section.

\paragraph*{Acknowledgements.}
We thank Jürgen Jost and Slava Matveev for many discussions, Dirk Hofmann and Walter Tholen for their advice, Jean Goubault-Larrecq, Tom\'a\v{s} Jakl and Xiaodong Jia for comments on this paper, and an anonymous reviewer for their extensive and helpful report. Much of this work was done while all three authors were with the Max Planck Institute for Mathematics in the Sciences.

\section{The Hoare hyperspace monad $H$}\label{hyperspace}

The powerset monad on the category of sets is among the most elementary examples of monads. It has an analogue on the category of topological spaces, which we study in this section. But its best-known analogue is on metric spaces, where the Hausdorff metric equips the space of nonempty closed subsets of a bounded metric space with a metric, turning it into a metric space in its own right~\cite{hausdorff}. For a topological space $X$, which may not carry a metric, one version of the hyperspace of $X$ was introduced by Vietoris~\cite{vietoris} who equipped the set of closed subsets of $X$ with the \emph{Vietoris topology}. This construction yields an endofunctor of $\cat{Top}$ which preserves compactness and connectedness. The deep study of hyperspaces by Michael~\cite{michael} showed that the Vietoris topology, when restricted to the nonempty closed sets, is induced by the Hausdorff metric whenever the base space is a compact metric space. The results of Michael implicitly equip the Vietoris functor on the category of compact spaces with a monad structure.

The Vietoris topology is the minimal common refinement of the \emph{lower Vietoris topology} and the \emph{upper Vietoris topology}. The functor that assigns to a topological space the set of its closed subsets with the lower Vietoris topology has been introduced by Smyth \cite{smyth83}. That this endofunctor has a monad structure has been shown by Schalk \cite{schalk} under the almost inconsequential restriction to the category of $T_0$ spaces. She also studied the algebras of this monad \cite{schalk}. Several related topological results are due to Clementino and Tholen \cite{cletho}.

In this section, we discuss this functor, which assigns to a topological space $X$ the space of its closed subsets $HX$ equipped with the lower Vietoris topology, as well as its monad structure. Using terminology proposed by Goubault-Larrecq\footnote{See \href{http://projects.lsv.ens-cachan.fr/topology/?page_id=585}{http://projects.lsv.ens-cachan.fr/topology/?page\_id=585}, where this was proposed in generalization of the term Hoare powerdomain.}, we call $HX$ the \emph{Hoare hyperspace} of $X$ and $H$ the \emph{Hoare hyperspace monad}.
Our construction of $H$ reveals that $H$ is a double dualization monad, where the double dualization is with respect to the Sierpiński space; this facilitates a treatment of the monad multiplication and the monad axioms that is purely formal and free from topological considerations. After an analogous treatment of the monad of continuous valuations in \Cref{sec_valuations}, we show in \Cref{secmonadmorph} that there is a morphism of monads from continuous valuations to $H$ which takes every continuous valuation, and therefore in particular every $\tau$-smooth Borel probability measure, to its support.

In working with hyperspaces, there is a choice to make concerning the membership of the empty set. This choice is relatively inconsequential, in the sense that most results hold either way. While most of the works mentioned above have excluded the empty set, we do include it.

\subsection{Duality theory for closed sets and the Hoare hyperspace}

Let $X$ be a topological space. Then every closed set $C \subseteq X$ can be \emph{paired} with any open $U \subseteq X$ by defining
\begin{equation}
 \label{H_pairing}
 \pair{C}{U} \;\coloneqq\; \begin{cases}
	 1 & \text{if } C \intersection U \ne \varnothing, \\
	 0 & \text{if } C \intersection U = \varnothing.
	\end{cases}
\end{equation}
We will occasionally say ``$C$ hits $U$'' to indicate that $C \intersection U \ne \varnothing$, as commonly done in the literature on hyperspaces~\cite{hitandmiss}.

Intuitively, $C$ is analogous to a measure, and $U$ is analogous to an integrable function, so that the pairing is analogous to integration.
To investigate the properties of this pairing further, we identify the set of open subsets $\op(X)$ with the set of maps $X \to S$, where $S \coloneqq \{0,1\}$ is the \emph{Sierpiński space} with open sets $\left\{ \varnothing , \{1\}, \{0,1\} \right\}$. We denote this set of maps by $S^X$. Thus, for a closed set $C$, we have
\[
	\pair{C}{-} \: : \: S^X \to S.
\]
Intuitively, an $S$-valued function---which is the same thing as an open set---can be integrated to an element of $S$---which is the same thing as a truth value, indicating whether $C$ hits the open set or not. Due to the natural bijection $\op(X) \cong S^X$, we may call the elements of $S^X$ the open subsets of $X$, may denote them by $U,V \subseteq X$, and may use them interchangeably as subsets of $X$ and as continuous functions $X \to S$. Since $\op(X)$ is a complete lattice with respect to the inclusion order, the same holds for $S^X$, where the partial order is the pointwise order of functions.

We then have the following known characterization of closed sets.

\begin{prop}[{e.g.~\cite[Proposition~5.4.2]{escardo}}]
	\label{closedduality}
	For any topological space $X$, the above pairing establishes a natural bijection between:
	\begin{enumerate}
		\item Closed subsets of $X$,
		\item Maps $\phi : S^X \to S$ that fulfill the following two conditions:
			\begin{enumerate}
				\item \label{H_modularity} \emph{Linearity:} For all $U,V\in S^X$ we have
					\[
						\phi(U \cup V) = \phi(U) \lor \phi(V),
					\]
					and $\phi(\varnothing) = 0$.
				\item \label{H_scottcontinuity} \emph{Scott continuity:} For any directed net $(U_\lambda)_{\lambda \in \Lambda}$ in $S^X$, we have
					\[
						\phi \left( \bigcup_{\lambda \in \Lambda} U_\lambda \right) = \bigvee_{\lambda \in \Lambda} \phi(U_\lambda).
					\]
			\end{enumerate}
	\end{enumerate}
	Under this bijection, the inclusion order on closed sets corresponds exactly to the pointwise order on maps $S^X \to S$.
\end{prop}

We call condition \ref{H_modularity} ``linearity'' since it is analogous to the linearity of integration, where its first subcondition is analogous to additivity and its second to the commutation with scalar multiplication.

\begin{proof}
	Conditions~\ref{H_modularity} and~\ref{H_scottcontinuity} state the preservation of nullary, binary, and directed joins, respectively, of $\phi$ as a map of complete lattices $S^X \to S$. It is well-known that these are jointly equivalent to the preservation of all joins.

	Given a closed set $C \subseteq X$, the pairing map $\pair{C}{-} : S^X \to S$ preserves all joins, since $C$ hits a union of open sets if and only if it hits one of them. Therefore $\pair{C}{-}$ satisfies the stated conditions. Conversely given $\phi : S^X \to S$ satisfying the stated conditions, $\phi^{-1}(0)$ is a collection of open sets, and we can consider the closed set
	\begin{equation}
		\label{phi_support}
		C \;\coloneqq\; X \setminus \left( \bigcup \phi^{-1}(0) \right).
	\end{equation}
	We need to prove that these two constructions are inverses of each other.

	For a closed set $C$, $\pair{C}{-}^{-1}(0)$ consists of all open sets disjoint from $C$. Since their union is $X \setminus C$, the construction~\eqref{phi_support} recovers $C$. For given $\phi : S^X \to S$, the preservation of joins shows that $\phi(U) = 0$ if and only if $U \subseteq \bigcup \phi^{-1}(0)$. (In other words, $\phi^{-1}(0)$ is an order ideal closed under joins, and therefore principal.) But this condition holds equivalently if $U$ is disjoint from~\eqref{phi_support}, as was to be shown.

	It remains to be shown that both constructions are monotone. Given closed sets $C \subseteq D$, the pointwise $\pair{C}{-} \le \pair{D}{-}$ is obvious. Conversely, if $\phi, \psi : S^X \to S$ satisfy our conditions and are such that $\phi \le \psi$ pointwise, then $\psi^{-1}(0) \subseteq \phi^{-1}(0)$, and the containment of the associated closed sets follows from~\eqref{phi_support}.
\end{proof}

We will often use this characterization to \emph{define} a closed set in terms of how it pairs with opens, in which case we need to verify the relevant conditions~\ref{H_modularity}--\ref{H_scottcontinuity}, or merely the preservation of all joins.

\begin{deph}[Hoare hyperspace]
	Let $X$ be a topological space.
	\begin{enumerate}
		\item The Hoare hyperspace over $X$, denoted $HX$, is the set of closed subsets of $X$, or, equivalently, the set of maps $S^X \to S$ of \Cref{closedduality}.
		\item We equip $HX$ with the weakest topology which makes the pairing maps
			\[
				\pair{-}{U} \: : \: HX \longrightarrow S
			\]
			continuous for every $U \in S^X$.
	\end{enumerate}
\end{deph}
Thus the subbasic open sets of $HX$ are the subsets of the form $\pair{-}{U}^{-1}(1)$.

Defining $HX$ as a set of maps $S^X \to S$ with this topology is a double dualization procedure and illustrates the sense in which $H$ acquires the structure of a double dualization monad. Following functional-analytic terminology, one could call it the \emph{weak topology}, as done by Schalk~\cite[Section~1.1.5]{schalk}\footnote{With the minor difference that she excludes the empty set.}, while in the picture of closed sets, the topology on $HX$ is known as the \emph{lower Vietoris topology}.

Frequently, it will be convenient to pair the closed sets $C \in HX$ not with all open sets, but merely with a specified subset of the open sets. In our framework, a \emph{basis} of a topological space is a subset $\mathcal{B} \subseteq S^X$ such that every $U \in S^X$ is a join of a subset of $\mathcal{B}$.

\begin{lemma}\label{H_basis}
	Let $\mathcal{B} \subseteq S^X$ be a basis for the topology of $X$.
	\begin{enumerate}
		\item\label{H_basis_specialization} For given $C,D \in HX$, the following three statements are equivalent:
			\begin{enumerate}
				\item\label{H_B_pointwise} For all $U \in \mathcal{B}$, we have $\pair{C}{U} \le \pair{D}{U}$;
				\item\label{H_specialization} $C \le D$ in the specialization preorder on $HX$;
				\item\label{H_inclusion} $C \subseteq D$ as subsets of $X$.
			\end{enumerate}
		\item The topology on $HX$ is the weakest topology which makes all of the pairings
			\[
				\pair{-}{U} \: : \: HX \longrightarrow S
			\]
			continuous for $U \in \mathcal{B}$.
	\end{enumerate}
\end{lemma}

\begin{proof}\hfill
\begin{enumerate}
\item Since the pairing maps preserve all joins in the second argument, condition~\ref{H_B_pointwise} holds if and only if it holds with $S^X$ in place of $\mathcal{B}$. Then the equivalence of~\ref{H_B_pointwise} and~\ref{H_inclusion} is the final part of \Cref{closedduality}. Replacing $\mathcal{B}$ by $S^X$, it is also clear that~\ref{H_B_pointwise} is equivalent to~\ref{H_specialization}, since~\ref{H_B_pointwise} then states exactly that every neighborhood of $C$ is a neighborhood of $D$, which is the definition of the specialization preorder.
\item We need to show that, if these maps are continuous, then so is every $\pair{-}{V}$ for $V \in S^X$. But this is an instance of the fact that (arbitrary) suprema of continuous maps to $S$ are continuous. \qedhere
\end{enumerate}
\end{proof}

\begin{remark}\label{uppervietoris}
The sets of the form $\pair{-}{U}^{-1}(1)$, which generate the lower Vietoris topology, are sometimes denoted by $\Hit(U)$ \cite{cletho}.
The set of closed subsets of a topological space is more commonly equipped with the full Vietoris topology \cite{vietoris}. 
The latter has as subbasic open sets not only all the sets $\Hit(U)$ for open $U$, but also the sets
$$
\mathrm{Miss}(C) \;\coloneqq\; \{D\in HX : D \intersection C = \varnothing \}
$$
for closed $C\subseteq X$. (One can say that ``$D$ \emph{misses} $C$'' if $D\in\mathrm{Miss}(C)$.) 
Intuitively, the lower Vietoris topology relates to the full Vietoris topology as the topology of lower semicontinuity on $\R$, with generating open sets of the form $(a,\infty)$, relates to the usual topology of $\R$. We use the lower Vietoris topology instead of the full Vietoris topology for two reasons: firstly, because the lower Vietoris topology of $HX$ parallels exactly the one on the space of continuous valuations $VX$ as constructed in \Cref{val_topology}; secondly, because, as a consequence of this correspondence, taking the support of a continuous valuation results in a continuous map $VX\to HX$ with respect to these topologies (see \Cref{precontsupport}). Indeed the full Vietoris topology would not make this map continuous, as \Cref{suppislsc} shows, but rather merely lower semicontinuous.
\end{remark}

The following observation of Hoffmann \cite[Example~2.3(a)]{hoffmann} makes the lower Vietoris topology more concrete, but needs to be read with care as it is one of the few statements which have no analogue for continuous valuations.
We denote by $\downarrow\!\! \{ C \}$ the principal downset generated by $C$, the set of closed subsets of $C$.

\begin{lemma}
 The topology on $HX$ is generated by complements of sets of the form $\downarrow\!\! \{C\}$ for $C \in HX$.
\end{lemma}
\begin{proof}
	Every $U \in S^X$ is of the form $U = X \setminus C$ for some closed set $C$, and the open set $\pair{-}{U}^{-1}(1)$ associated to $U$ is the complement of $\downarrow\!\! \{C\}$.
\end{proof}

Thus the lower Vietoris topology is determined by the inclusion order of closed sets, and coincides with the \emph{upper topology} on the set of closed sets~\cite[Definition~O.5.4]{continuous}. 
It is worth noting that the point $X \in HX$, being contained in every nonempty open set, is dense in $HX$. At the other extreme, the only open set which contains the point $\varnothing \in HX$ is the entire space $HX$, which implies that $\varnothing$ is in the closure of every nonempty set.

\begin{prop}
	\label{HX_sober}
	For every topological space $X$, its Hoare hyperspace $HX$ is sober.
\end{prop}

The argument is the same in spirit as Heckmann's proof of sobriety of the space of valuations~\cite[Proposition~6.1]{heckmann95}.

\begin{proof}
It is a standard fact that limits of sober spaces (in $\cat{Top}$) are sober~\cite[Theorem~8.4.13]{nonhausdorff}.	It is therefore enough to exhibit $HX$ as a limit of sober spaces. A diagram which achieves this can be directly read off the characterization of \Cref{closedduality}, which characterizes $HX$ as a subset of $S^{S^X}$ equipped with the product topology (which, by the same theorem, is sober). This subset is characterized by the given equations, where imposing each equation amounts to equalizing a pair of continuous maps to $S$. An inequality $r \le s$ can equivalently be considered as the equation $r \lor s = s$ and $\lor : S \times S \to S$ is continuous. It is therefore enough to prove that each side of each instance of those conditions depends continuously on $\phi \in S^{S^X}$. This is obvious in all cases, where for $\bigvee_{\lambda \in \Lambda} \phi(U_\lambda)$ one needs to use the fact that $\bigvee : S^I \to S$ is continuous.
\end{proof}

Note the difference between $\cl(\{A\})$, the closure of the singleton $\{A\}$ in $HX$, and $\cl(A)$, the closure of $A$ in $X$. 

\subsection{Functoriality}
\label{H_fun}

The construction of $H$ as a functor $H : \cat{Top} \to \cat{Top}$ is due to Smyth~\cite{smyth83}. We here treat the functoriality based on our double dualization approach.
 
\begin{deph}[Pushforward]
	Let $f:X\to Y$ be continuous and $C \in HX$. The \emph{pushforward} $f_\sharp C \in HY$ is defined through the pairing with any $V \in S^Y$ as\footnote{Here our notation is such that $V : X \to S$ is interpreted as a function. With $V \subseteq X$ as an open set, we would have to write $f^{-1}(V)$ in place of $V \circ f$.}
\begin{equation}
	\label{H_functoriality}
	\pair{f_\sharp C}{V} \coloneqq \pair{C}{V \circ f}.
\end{equation}
\end{deph}

It is easy to see that $\pair{f_\sharp C}{-} : S^Y \to S$ satisfies the conditions of \Cref{closedduality}. By the analogy with integration, this makes the change of variables formula into a definition.
In fact, this is quite generally how functoriality works for double dualization functors: the action on morphisms can be defined in terms of the pairing with respect to the first step of dualization, where one simply composes by the given morphism $f$.

In the picture of closed sets, $f_\sharp C$ is the closure of the image of $C$,
\[
	f_\sharp C = \cl\!\big(f(C)\big),
\]
since an open set $V$ is disjoint from $\cl\!\big(f(C)\big)$ if and only if $f^{-1}(V)$ is disjoint from $C$.

The definition~\eqref{H_functoriality} makes it immediate that $f_\sharp : HX \to HY$ is continuous, a fact that, although simple, is not obvious in the picture of closed sets. A similar comment applies to the following functoriality statement.

\begin{lemma}
 \label{H_is_functor}
 Let $f:X\to Y$ and $g:Y\to Z$ be continuous maps. Then $(g \circ f)_\sharp = g_\sharp \circ f_\sharp$.
\end{lemma}

\begin{proof}
 Let $C\in HX$. Then we have
 \begin{align*}
   \pair{(g \circ f)_\sharp (C)}{U} \;&=\; \pair{C}{U \circ (g \circ f)} \;=\; \pair{C}{(U \circ g) \circ f} \\
   &=\; \pair{f_\sharp C}{U \circ g} \;=\; \pair{g_\sharp(f_\sharp C)}{U} , 
 \end{align*} 
 for any $U\in S^Z$, which implies the claim.
\end{proof}

The preservation of identities is trivial.
We have therefore obtained a functor $H:\cat{Top}\to\cat{Top}$ which assigns to each topological space $X$ its Hoare hyperspace $HX$, and to each continuous map $f:X\to Y$ the continuous map $f_\sharp:HX\to HY$. We call $H$ the \emph{Hoare hyperspace functor}. The functor $H$ is also a 2-functor in the sense of \Cref{2cat}.

\begin{lemma}\label{H2monad}
 Let $X$ and $Y$ be topological spaces and let $f,g:X\to Y$ be continuous maps with $f\le g$. Then $f_\sharp \le g_\sharp$. In other words, $H$ preserves 2-cells, making it into a 2-functor. 
\end{lemma}

\begin{proof}
By \Cref{ordopen}, we have $V \circ f \le V \circ g$ with respect to the pointwise order on $S^X$ for every $V \in S^Y$.
Therefore
\[
\pair{f_\sharp C}{V} = \pair{C}{V \circ f} \le \pair{C}{V \circ g} = \pair{g_\sharp C}{V},
\]
which, by \Cref{2cell_def} and \Cref{H_basis}, means $f_\sharp \le g_\sharp$.
\end{proof}

\subsection{Monad structure}

We equip the functor $H$ with a monad structure, making it into a topological analogue of the powerset monad. While doing so, our double dualization perspective will be handy.

\subsubsection{Unit}

\begin{deph}
Let $X$ be a topological space. We define the map $\sigma : X \to HX$ by
\begin{equation}
	\label{sigma_def}
	\pair{\sigma(x)}{U} \; \coloneqq \; U(x) =
 		\begin{cases}
			1 & \textrm{if } x \in U, \\
			0 & \textrm{if } x \not\in U,
		\end{cases}
\end{equation}
for every $U \in S^X$.
\end{deph}

It is obvious that the preservation of joins in the second argument holds, since a point $x$ is contained in a union of open sets if and only if it is contained in one of them. $\sigma$ is continuous by definition, since every $x \mapsto \pair{\sigma(x)}{U}$ is continuous as a map $X \to S$.

As the definition suggests, $\sigma(x)$ is the possibilistic analogue of the Dirac measure at $x$. In the picture of closed sets, we have
\[
	\sigma(x) = \cl(\{x\}),
\]
which is clear since an open set $U$ does not contain $x$ if and only if $\cl(\{x\})$ is disjoint from $U$. We now turn to naturality.

\begin{lemma}
 Let $X$ and $Y$ be topological spaces and let $f:X\to Y$ be continuous. Then the following diagram commutes.
 \begin{equation*}
  \begin{tikzcd}
   X \ar{r}{f} \ar{d}{\sigma} & Y \ar{d}{\sigma} \\
   HX \ar{r}{f_\sharp} & HY
  \end{tikzcd}
 \end{equation*}
\end{lemma}

\begin{proof}
 For $x\in X$ and $U\in S^Y$, we have
$$
 \pair{f_\sharp(\sigma(x))}{U} \;=\; \pair{\sigma(x)}{U \circ f} \;=\; (U \circ f)(x)) \;=\; U(f(x)) \;=\; \pair{\sigma(f(x))}{U} . \qedhere
$$
\end{proof}

Hence we have a natural transformation $\sigma:\id\Rightarrow H$ between endofunctors of $\cat{Top}$.

\subsubsection{Topological properties of the unit map}\label{geounit}

We state conditions on $X$ under which $\sigma:X\to HX$ is a homeomorphism onto its image (a subspace embedding), and consider when $\sigma(X)\subseteq  HX \setminus \{\varnothing\}$ is closed. The \emph{Kolmogorov quotient} of a space $X$ is the quotient space $X/\sim$ where any two points of $X$ that have the same open neighborhoods are identified (\Cref{2cat}).

\begin{prop}\label{imagesigma2}
Let $X$ be a topological space and consider the map $\sigma:X\to HX$. Then $\sigma(X)$ is the Kolmogorov quotient of $X$ with respect to $\sigma : X \to \sigma(X)$ as the quotient map. In particular, $\sigma$ is a homeomorphism onto its image if and only if $X$ is $T_0$.
\end{prop}

\begin{proof} 
By definition~\eqref{sigma_def}, $\sigma$ is injective if and only if no two distinct points of $X$ have the same open neighborhoods, hence if and only if $X$ is $T_0$. Thus the statement holds at the set-theoretical level and we can identify the points of $\sigma(X)$ with the equivalence classes of points of $X$ with respect to topological indistinguishability.

It remains to be shown that the subspace topology on $\sigma(X)$ induced from $HX$ is the quotient topology with respect to $\sigma : X \to \sigma(X)$ as the quotient map. Since taking the initial topology commutes with the passage to subspaces, the topology on $\sigma(X)$ is the weakest topology which makes the maps
\[
\sigma(X) \longrightarrow S, \qquad C \longmapsto \pair{C}{U}
\]
for $U \in S^X$ continuous. Since such an open set contains precisely those equivalence classes whose elements are in $U$, the topology on $\sigma(X)$ is the topology of the Kolmogorov quotient.
\end{proof}

It may also be interesting to know whether $\sigma(X) \subseteq HX$ is a closed subspace.
Clearly $\cl(\sigma(X))$ always contains $\varnothing$, and therefore $\sigma(X)$ is not closed in $HX$ (unless $X$ is empty).
But it is still meaningful to ask when $\cl(\sigma(X)) = \sigma(X) \cup \{\varnothing\}$, or equivalently when $\sigma(X)$ is closed in $HX \setminus \{\varnothing\}$.\footnote{Where the equivalence is an immediate consequence of \Cref{H_basis}\ref{H_basis_specialization} (assuming $X \neq \varnothing$).}
We next state conditions for this to be the case, and note that it is not true in general, even if $X$ is $T_1$ or sober.

\begin{prop}\label{closuresigma}
Let $X$ be a nonempty topological space and consider a closed set $C\subseteq X$. Then the following two conditions are equivalent:
\begin{enumerate}
\item $C \in \cl(\sigma(X))$;
\item\label{opens_intersect} For every finite collection of open sets $U_1,\dots,U_n \subseteq X$ that hit $C$, their intersection $U_1\cap\dots\cap U_n$ is nonempty (but possibly disjoint from $C$).
\end{enumerate}
\end{prop}

\noindent Note that~\ref{opens_intersect} is trivially true for $C = \varnothing$ and when $C$ is the closure of a singleton.

\begin{proof}
$C$ is \emph{not} in the closure of $\sigma(X)$ if and only if this is witnessed by some basic open set. By definition of the topology, this means that there are $U_1, \ldots, U_n \in S^X$ such that
\[
\pair{C}{U_i} = 1, \quad \forall i=1,\ldots,n,
\]
but such that for every $x \in X$ there is $i$ with
\[
\pair{\sigma(x)}{U_i} = U_i(x) = 0.
\]
The former means exactly that every $U_i$ hits $C$, and the latter that $U_1 \cap \ldots \cap U_n = \varnothing$.
\end{proof}

\begin{eg}\label{dense_image}
Let $X$ be any nonempty space in which finite intersections of nonempty open sets are nonempty. Then it follows that $\sigma(X)$ is dense in $HX$, since condition \ref{opens_intersect} is always satisfied. For example, $X$ could be any infinite set equipped with the cofinite topology (which is even $T_1$). As another example, take any space $X$ which has a dense point, or equivalently a greatest element in the specialization preorder. More concretely, every $X = HY$ for any $Y$ is a sober space with dense point $Y \in HY$, and therefore $\sigma(HY)$ is dense in $HHY$.
\end{eg}

\begin{cor}\label{Haus_closed}
Let $X$ be a Hausdorff space. Then $\sigma(X) \union \{\varnothing\}$ is closed in $HX$.
\end{cor}

\begin{proof}
If $C$ contains at least two different points, then $C$ hits disjoint open neighborhoods $U_1$ and $U_2$ which separate these points.
\end{proof}

There are non-Hausdorff spaces $X$ for which $\sigma(X) \cup \{\varnothing\}$ is closed. An example is given by the unit interval $X = [0,1]$ equipped with the upper open topology, whose open sets are the intervals of the form $(a,1]$, then we even have $HX = \sigma(X) \cup \{\varnothing\}$. More interestingly, in the $T_1$ case also the converse to \Cref{Haus_closed} is true.

\begin{prop}
Let $X$ be a space which is $T_1$, but not Hausdorff. Then $\sigma(X) \union \{\varnothing\}$ is not closed in $HX$.
\end{prop}

\begin{proof}
 Let $x,y$ be distinct points of $X$ such that for any open neighborhoods $U \ni x$ and $V \ni y$, we have $U \cap V \neq \varnothing$. We can find such points since $X$ is not Hausdorff. The subset $C\coloneqq \{x,y\}$ is closed since $X$ is $T_1$. Since $x$ and $y$ are distinct and singletons are closed, $C$ is not in the image of $\sigma$. Let $U_1,\dots, U_n$ be a finite collection of open sets such that $C$ hits all of them. Then each of them contains $x$, or $y$, or both. Reorder the $U_i$ in such a way that $U_1,\dots, U_k$ contain at least $x$, and $U_{k+1},\dots, U_n$ contain at least $y$. Then 
 $$
 \bigintersection_{i=1}^n U_i = \left(\bigintersection_{i=1}^k U_i\right) \intersection \left(\bigintersection_{j=k+1}^n U_j\right).
 $$
 We have that $\intersection_{i=1}^k U_i$ is an open neighborhood of $x$, and  $\intersection_{j=k+1}^n U_j$ is an open neighborhood of $y$, and these two neighborhoods have nonempty intersection. By \Cref{closuresigma}, then, $C$ is in the closure of $\sigma(X)$. Since $C$ is not in $\sigma(X)$, it follows that $\sigma(X) \union \{\varnothing\}$ is not closed.
\end{proof}

\subsubsection{Multiplication}

The definition of the monad multiplication is where our double dualization picture is particularly useful, since there is a very simple definition in terms of the pairing map $\pair{-}{U} \in S^{HX}$ for $U \in S^X$.

\begin{deph}\label{defmultH}
Let $X$ be a topological space. We define the map $\U:HHX\to HX$ on any $\mathcal{C} \in HHX$ by
\begin{equation}
\label{mult_def}
\pair{\U\mathcal{C}}{U} \;\coloneqq\; \pair{\mathcal{C}}{\pair{-}{U}}
\end{equation}
for every $U \in S^X$.
\end{deph}

Since the pairing is join-preserving in the second argument, it is clear that~\eqref{mult_def} satisfies the requirements of \Cref{closedduality} which guarantee that $\U(\mathcal{C}) \in HHX$.
In the picture of closed sets, $\U$ assigns to each closed set of closed sets the closure of their union,
\[
	\U \mathcal{C} \;=\; \cl \left(\bigcup \, \mathcal{C} \right) \;=\; \cl \left( \bigcup_{C\in\mathcal{C}} C \right) .
\]
This is because, by definition~\eqref{mult_def}, $\U\mathcal{C}$ is disjoint from $U$ if and only if $\mathcal{C}$ is disjoint from the set of all closed sets that hit $U$.

To show that $\U$ is continuous, we only need to verify that its composition with any $\pair{-}{U}$ is continuous, which holds by definition. We turn to naturality. 

\begin{prop}
\label{unionnatural}
 Let $X$ and $Y$ be topological spaces and let $f:X\to Y$ be continuous. Then the following diagram commutes. 
 \begin{equation*}
  \begin{tikzcd}
   HHX \ar{r}{f_{\sharp\sharp}} \ar{d}{\U} & HHY \ar{d}{\U} \\
   HX \ar{r}{f_\sharp} & HY
  \end{tikzcd}
 \end{equation*}
\end{prop}

In terms of closed sets, this amounts to the statement that taking images commutes with unions in the case of discrete spaces, and to a similar statement involving closures in general.

\begin{proof}
	Let $\mathcal{C}\in HHX$ and $V \in S^Y$. We unfold the definition to see
 \begin{align*}
   \pair{\U(f_{\sharp\sharp}\mathcal{C})}{V} \;&=\; \pair{f_{\sharp\sharp} \mathcal{C}}{\pair{-}{V}} \;=\; \pair{\mathcal{C}}{\pair{-}{V} \circ f_\sharp} =\; \pair{\mathcal{C}}{\pair{f_\sharp(-)}{V}} \\
   &=\; \pair{\mathcal{C}}{\pair{-}{V \circ f}}  \;=\; \pair{\U\mathcal{C}}{V \circ f} \;=\; \pair{f_\sharp(\U\mathcal{C})}{V} . \qedhere
 \end{align*}
\end{proof}

\subsubsection{Monad axioms}

\begin{prop}
\label{monaddiagramsH}
 Let $X$ be a topological space. Then the following three diagrams commute.
 \[
  \begin{tikzcd}
   HX \ar{r}{\sigma} \idar{dr} & HHX \ar{d}{\U}  \\
   & HX
  \end{tikzcd}
  \qquad
  \begin{tikzcd}
   HX \ar{r}{\sigma_\sharp} \idar{dr} & HHX \ar{d}{\U} \\
   & HX
  \end{tikzcd}
  \qquad
  \begin{tikzcd}
   HHHX \ar{r}{\U_\sharp} \ar{d}{\U } & HHX \ar{d}{\U} \\
   HHX \ar{r}{\U} & HX
  \end{tikzcd}
 \]
\end{prop}

In terms of closed sets, these diagrams are the topological analogs of basic facts of set theory. For sets, the first two unitality diagrams state that the union over a singleton set of sets is the set itself, and that the union of singletons is the set whose elements are the respective singletons. The associativity diagram states that taking unions is an associative operation. In our double dualization frameworks, all three proofs are mere unfoldings of definitions.

\begin{proof}[Proof of \Cref{monaddiagramsH}]
	We start with the first diagram, left unitality. Let $C\in HX$ and $U \in S^X$. Then
	\[
		\pair{\U\sigma(C)}{U} \;=\; \pair{\sigma(C)}{\pair{-}{U}} \;=\; \pair{-}{U}(C) \;=\; \pair{C}{U} .
	\]
	Right unitality works similarly,
  	\begin{align*}
		\pair{\U\sigma_\sharp(C)}{U} \;&=\; \pair{\sigma_\sharp C}{\pair{-}{U}} \;=\; \pair{C}{\pair{-}{U} \circ \sigma} \;=\; \pair{C}{\pair{\sigma(-)}{U}} \\
		&=\; \pair{C}{U(-)} \;=\; \pair{C}{U}.
  	\end{align*}
	It remains to consider the associativity diagram. For $\mathcal{K}\in HHHX$ and $U \in S^X$ we get
	\begin{align*}
		\pair{\U(\U_\sharp \mathcal{K})}{U} \;&=\; \pair{\U_\sharp \mathcal{K}}{\pair{-}{U}} \;=\; \pair{\mathcal{K}}{\pair{-}{U}\circ \U} \\
		&=\; \pair{\mathcal{K}}{\pair{\U(-)}{U}} \;=\; \pair{\mathcal{K}}{\pair{-}{\pair{-}{U}}} \\
		&=\; \pair{\U\mathcal{K}}{\pair{-}{U}} \;=\; \pair{\U(\U\mathcal{K})}{U} .
		\qedhere
	\end{align*}
\end{proof}

\noindent We have proven the following statement.

\begin{thm}
 The triple $(H,\sigma,\U)$ is a monad on $\cat{Top}$. 
\end{thm}

We call $(H,\sigma,\U)$, or just $H$, the \emph{Hoare hyperspace monad}. By \Cref{H2monad}, $H$ is a strict 2-monad for the 2-categorical structure of $\cat{Top}$ given in \Cref{2cat}.

As far as we know, this monad was introduced by Schalk \cite[Section~6.3.1]{schalk}, with the inessential difference that the empty set was excluded from the hyperspace.

\subsection{Algebras of $H$}\label{Halg}

There is a characterization of the Eilenberg-Moore algebras of the monad $H$, which is, as far as we know, also due to Schalk~\cite[Section~6.3.1]{schalk}.
We review the main results. An additional reference, which also discusses related constructions, is Hoffmann's earlier article~\cite{hoffmann}.

Before we begin, it is worth noting that the metric or Lawvere-metric counterpart~\cite{catgromovhausdorff} of the monad $H$, becomes a Kock-Zöberlein monad~\cite{kzkock,kzzoberlein} upon considering it as a 2-monad on a strict 2-category. This means that whenever a topological space admits an $H$-algebra structure, then this structure is unique. This phenomenon is a property-like structure~\cite{kzkellylack}. Nevertheless, we will see that not every morphism of $\cat{Top}$ between $H$-algebras is a morphism of $H$-algebras.

It is well-known that the algebras of the powerset monad on $\cat{Set}$ are the complete join-semilattices. The $H$-algebras are their topological cousins. An $H$-algebra is, by definition, a pair $(A,a)$ consisting of a topological space $A$ and a continuous map $a: HA \to A$ such that the following two diagrams commute.

\begin{equation}\label{Halgdiag}
 \begin{tikzcd}
A \arrow{r}{\sigma} \arrow[equals]{dr} & HA \arrow{d}{a} \\
{} & A
\end{tikzcd}
\hspace{40pt}
\begin{tikzcd}
HHA \arrow{r}{a_\sharp} \arrow{d}{\U} & HA \arrow{d}{a} \\
HA \arrow{r}{a} & A
\end{tikzcd}
\end{equation}
We refer to these diagrams as the unit diagram and the algebra diagram. A first observation is that every $H$-algebra $A$ is a $T_0$ space, since $\sigma$ must be injective by the unit diagram, but $\sigma$ identifies topologically indistinguishable points (\Cref{imagesigma2}).\footnote{Non-$T_0$ spaces can still be pseudoalgebras, if we consider $H$ as a 2-monad on a 2-category.} In fact, since every $HA$ is sober (\Cref{HX_sober}) and retracts of sober spaces are sober,\footnote{See~\cite[Exercise~O-5.16]{continuous}, or note that every retract arises as an equalizer and use the fact that sober spaces are closed under limits~\cite[Theorem~8.4.13]{nonhausdorff}.} the unit diagram shows that every $H$-algebra is sober.

\begin{deph}[Topological complete join-semilattice]
\label{top_lattice}
 A topological complete join-semilattice is a complete lattice $L$ equipped with a sober topology whose specialization preorder coincides with the lattice order and whose join map $\join:L\times L\to L$ is continuous.
\end{deph}

Since the structure of such a lattice is completely determined by its topology, we may consider these lattices as a particular class of topological spaces. Hoffmann~\cite{hoffmann} calls them \emph{essentially complete $T_0$ spaces} and also \emph{$T_0$ topological complete sup-semilattices}, while Schalk~\cite{schalk} calls them \emph{unital inflationary topological semilattices}.
These spaces admit several equivalent characterizations~\cite[Theorem~1.8]{hoffmann}. They are precisely the $H$-algebras. 

\begin{thm}[Schalk]\label{halgebras}
	The category of $H$-algebras is isomorphic to the subcategory of $\cat{Top}$ whose objects are topological complete join-semilattices, with algebra maps given by the lattice join, and with continuous maps that preserve arbitrary joins as morphisms of algebras.
\end{thm}

This result has been claimed by Hoffmann~\cite[Theorem 2.6]{hoffmann} in the form of a monadicity statement, but apparently without explicit proof. As far as we know, the proof is essentially due to Schalk \cite[Section~6.3]{schalk}, culminating in Theorems~6.9 and 6.10 therein, which state this characterization in the full subcategory of sober spaces. This is not a restriction since only sober spaces can be $H$-algebras to begin with.

We now present a proof of \Cref{halgebras} which is more direct than Schalk's, starting with some auxiliary results. Since these statements are quite specific to this particular monad, it is more convenient to work in terms of closed sets rather than with the double dualization framework.

\begin{lemma}\label{algtolat}
 Let $(A,a)$ be an $H$-algebra. Then $a : HA \to A$ assigns to every closed set a join with respect to the specialization preorder of $A$.
\end{lemma}

\noindent The proof only uses the fact that $a$ is a continuous retraction of $\sigma : A \to HA$.

\begin{proof}
	Let $C\in HA$. Since $a$ is continuous, it is monotone for the specialization preorder. For every $x\in C$ we have $\sigma(x) = \cl(\{x\}) \subseteq C$ and therefore, by the unit condition for algebras,
$$
x \;=\; a(\sigma(x)) \;\le\; a(C) .
$$
Hence $a(C)$ is an upper bound for $C$. On the other hand, let $u$ be any upper bound for $C$. Then $C\subseteq \sigma(u)$, which implies
$$
a(C) \;\le\; a(\sigma(u)) = u.
$$
Therefore $a(C)$ is a least upper bound for $C$, as was to be shown.
\end{proof}

\begin{lemma}[Lemma~1.5 in \cite{schalk}]\label{joinclosure}
 Let $X$ be a topological space. Then a subset $S \subseteq X$ admits a supremum in the specialization preorder if and only if $\cl(S)$ does, in which case they coincide.
\end{lemma}
\begin{proof}
 A point $x \in X$ is an upper bound for $S$ if and only if $S\subseteq \;\downarrow\!{x}=\cl(\{x\})$. Since the latter set is closed, $S\subseteq \;\downarrow\!{x}$ if and only if $\cl(S) \subseteq \;\downarrow\!{x}$. So the set of upper bounds of $S$ is equal to the set of upper bounds of $\cl(S)$. If this set admits a lowest element, this is the least upper bound of both $S$ and $\cl(S)$.
\end{proof}

\noindent The following lemma is somewhat converse to \Cref{algtolat}.

\begin{lemma}\label{lattoalg}
 Let $A$ be a $T_0$ space whose specialization preorder is a complete lattice. Suppose that the join map on closed sets $\bigjoin: HA \to A$ is continuous. Then $(A,\bigjoin)$ is an $H$-algebra.
\end{lemma}
\begin{proof}
	We need the continuity of $\bigjoin$ to ensure that it is a morphism in $\cat{Top}$. We have to verify the commutativity of the diagrams \Cref{Halgdiag}. The unit diagram requires that, for each $x\in A$, we have $\bigjoin \cl(\{x\}) = x$, which holds by \Cref{joinclosure}.
	The algebra diagram requires that, for each $\mathcal{C}\in HHA$, we have $\bigjoin \, \U(\mathcal{C}) = \bigjoin \left(\bigjoin_\sharp \mathcal{C}\right)$.
	Using \Cref{joinclosure} and the lattice-theoretic fact that the join of an arbitrary union of sets is the join of the individual joins,
	we have that 
	\begin{align*}
		\bigjoin \, \U(\mathcal{C}) &=\; \bigjoin \left(\cl \left( \bigcup \mathcal{C} \right)\right) \;=\; \bigjoin \left( \bigcup \mathcal{C}\right) \;=\; \bigjoin \left\{ \bigjoin C \: : C \in \mathcal{C} \right\} \\
		&=\; \bigjoin \cl \left( \left\{ \bigjoin C \: : C \in \mathcal{C} \right\} \right) \;=\; \bigjoin \left( {\bigjoin}_\sharp \mathcal{C}\right) . \qedhere
	\end{align*}
\end{proof}

\begin{lemma}[Lemma~II.1.9 in~\cite{stonespaces}]\label{openscottopen}
	Let $X$ be a sober topological space. Then the specialization preorder of $X$ has directed joins, and every open set $U \subseteq X$ is Scott-open for the specialization preorder.
\end{lemma}

\begin{proof}[Sketch of proof]
If $S \subseteq X$ is directed in the specialization order, then $\cl(S)$ is irreducible and therefore the closure of a unique point, which can be seen to be the supremum. This construction of directed joins makes the second statement obvious, since every open set is upward closed.
\end{proof}

\begin{lemma}\label{bothjoins}
Let $L$ be a sober topological space. Then the specialization preorder of $L$ has binary joins if and only if it has all joins. Furthermore, in that case, the binary join map $\join:L\times L\to L$ is continuous if and only if the join map for closed sets $\bigjoin : HL \to L$ is continuous.
\end{lemma}

\begin{proof}
As in \Cref{uppervietoris}, we also write $\Hit(U) := \pair{-}{U}^{-1}(1)$ for the subbasic open subset of $HL$ associated to an open $U \subseteq L$, consisting of all those closed sets that hit $U$.

If binary (and hence finitary) joins exist, then arbitrary joins exist by \Cref{openscottopen}. The converse is trivial. We thus only need to show that the join map $\bigjoin : HL \to L$ is continuous if and only if the binary join map $\join : L \times L \to L$ is.
	
Suppose that $\bigjoin : HL \to L$ is continuous. By \Cref{joinclosure}, it suffices to show that the map $\phi : L \times L \to HL$ defined by $\phi(x,y) = \cl(\{x,y\})$ is continuous. Consider any subbasic open set $\Hit(U)$ for open $U \subseteq L$. It is then enough to prove that the preimage $\phi^{-1}(\Hit(U))$ is open. In fact
\begin{align*}
\phi^{-1}(\Hit(U)) \;&=\; \{(x,y)\,|\,\cl(\{x,y\}) \intersection U \ne \varnothing \} \;=\; \{(x,y)\,|\,\{x,y\} \intersection U \ne \varnothing \} \\
  &=\; \{(x,y)\,|\,x\in U \} \union  \{(x,y)\,|\,y\in U \} \;=\; (U\times L) \union (L\times U) ,
\end{align*}
which is open in $L\times L$.

Suppose that the binary join map is continuous. We show that for every open set $U \subseteq L$, every $C \in \bigjoin^{-1}(U)$ has an open neighborhood contained in $\bigjoin^{-1}(U)$. Since $U$ is Scott-open by \Cref{openscottopen} and $\bigjoin C \in U$, there exists a finite set $\{x_1 , ..., x_n\} \subseteq C$ such that $x_1 \join \ldots \join x_n \in U$. Since the $n$-ary join map $L^n \to L$ is continuous by continuity of the binary one, there exist open neighborhoods $V_i \ni x_i$ such that for all $y_i \in V_i$ we have $y_1 \join \ldots \join y_n \in U$ as well. Consider the open set $W \coloneqq \bigcap_{i=1}^n \Hit (V_i)$. We have $C \in W$ by construction, and it is easy to see that $W \subseteq \bigjoin^{-1} (U)$ since $U$ is an upper set.
\end{proof}

\noindent We are now ready to prove the theorem. 

\begin{proof}[Proof of \Cref{halgebras}]
	By \Cref{algtolat} and \Cref{joinclosure}, we know that for an $H$-algebra $A$, every subset of $A$ must have a supremum, that is $A$ is a complete lattice in the specialization preorder. The Hoare hyperspace $HA$ is sober (\Cref{HX_sober}) and the map $\bigjoin: HA \to A$ must be continuous. By the unit condition of \Cref{Halgdiag}, it follows that $A$ is a retract of a sober space and therefore sober~\cite[Exercise~O.5.16]{continuous}. By \Cref{bothjoins}, the binary join is continuous too. Therefore $A$ is a topological complete join-semilattice and the algebra map is the join of closed sets.

Conversely, suppose that $A$ is a topological complete join-semilattice. By \Cref{bothjoins}, the join map of closed sets $\bigjoin: HA \to A$ is continuous. Using \Cref{lattoalg}, we conclude that $(A,\bigjoin)$ is an $H$-algebra.

To complete the proof, suppose that $A$ and $B$ are $H$-algebras. A morphism $m$ between them is, by definition, a continuous map such that the following diagram commutes.
\begin{equation*}
\begin{tikzcd}
HA \arrow{d}{\bigjoin} \arrow{r}{Hm} & HB \arrow{d}{\bigjoin} \\
A \arrow{r}{m} & B
\end{tikzcd}
\end{equation*}
Such maps $m$ are precisely those that preserve arbitrary suprema by \Cref{joinclosure}.
\end{proof}

We conclude this subsection with a remark. In contradiction to a claim by Schalk \cite[Sections~6.3 and 6.3.1]{schalk}, not every sober space whose specialization order is a complete lattice is a topological complete join-semilattice. In other words, for a sober space $X$ whose specialization order is a complete lattice, the continuity of the join map of closed sets $\bigjoin: HL \to L$, or equivalently of the binary join map $\join: L \times L\to L$, is not guaranteed.
A counterexample seems to be given by Hoffmann~\cite[Example~5.5 combined with Lemma~1.5]{hoffmann}, however it is based on what appears to be an incorrect reference (reference [5] therein). We give a concrete counterexample, based on Hoffmann's approach.

\begin{eg}[A sober space whose specialization preorder is a complete lattice, but whose binary join map is not continuous]
	Let $X$ be a $T_1$ space that is sober but not $T_2$. For example, $X$ could be the set $\N\union\{a,b\}$, where the open sets are given by all subsets of $\N$ and the sets of the form $\{a\}\union\N\setminus F$ and $\{b\}\union\N\setminus F$ and $\{a,b\}\union\N\setminus F$ where $F \subseteq \N$ is finite. Since $X$ is $T_1$, its specialization preorder is the discrete order.

Consider the set $X^* \coloneqq X \sqcup \{ \bot, \top \}$ with the topology whose nonempty open subsets $W \subseteq X^*$ are those in the form $W = V \cup \{ \top \}$ where $V$ is an open subset of $X$. The specialization preorder of $X^*$ is a complete lattice where
$$ x\join y = \begin{cases}
	\bot & \text{if } x = y = \bot, \\
	x & \text{if } x = y \in X, \\
	\top & \text{otherwise},
\end{cases}
$$
and similarly for arbitrary joins. To see that $X^*$ is sober, let $K \subseteq X^*$ be a nonempty irreducible closed set. If $\top \in K$, then necessarily $K = X^* = \cl(\{\top\})$. Hence we can assume $K = C \cup \{ \bot\}$ for some closed $C \subseteq X$. Suppose that $C = C_1 \cup C_2$ for closed $C_1,C_2 \subseteq X$. Then $K = \left( C_1 \cup \{\bot\} \right) \cup \left( C_2 \cup \{\bot\} \right)$. Since $K$ is irreducible, we must have $C_1 = C$ or $C_2 = C$. Therefore $C$ is also irreducible. Since $X$ is sober, $C$ is the closure of a unique point in $X$ and so must be $K$ in $X^*$. We conclude that $X^*$ is sober.

Consider the open subset $\{\top\} \subseteq X^*$. We will show that the preimage
$$
	{\join}^{-1} \left( \{ \top \} \right) = \left\{ (x,y) \in X^* \times X^* \:\big|\: x \join y = \top \right\}
$$
is not open, and therefore that $\join: X^* \times X^* \to X^*$ is not continuous, by exhibiting $(x,y) \in \join^{-1} \left( \{ \top \} \right)$ which is not in the interior of this preimage. Since $X$ is not $T_2$, there are distinct points $x, y \in X$ such that any open neighborhoods $U \ni x$ and $V \ni y$ in $X$ intersect. We then have $x \join y = \top$ since $x \neq y$. Consider any basic neighborhood $U' \times V'$ of $(x,y)$ in $X^* \times X^*$. These necessarily have the form $U' = U \cup \{\top\}$ and $V' = V \cup \{\top\}$ for certain open $U,V \subseteq X$. Since $U \cap V \neq \varnothing$, we have $z \in X$ such that $(z,z) \in U' \times V'$, but clearly $(z,z) \not \in \join^{-1}(\{\top\})$. Hence $(x,y)$ is not an interior point of $\join^{-1}(\{\top\})$.

	If $X = \N \cup \{a,b\}$ as above, the failure of continuity of the binary join can be illustrated as follows. The sequence $(1,2,3,\ldots)$ converges to both $a$ and $b$ in $X^*$, but $(1 \join 1, 2 \join 2, 3 \join 3, \ldots)$, which is the same sequence $(1,2,3,\ldots)$, does not converge to $a \join b = \top$.
\end{eg}

\subsection{Products and projections}
\label{H_bimonoidal}

In this section, we show that $H$ is a commutative monad (\Cref{comm_monads}), or, equivalently, a symmetric monoidal monad, with respect to the Cartesian monoidal structure on $\cat{Top}$. We start by constructing a \emph{strength} transformation $X \times HY \to H(X \times Y)$.

Given an open $W : X \times Y \to S$ and $x \in X$, we can consider $W(x,-) : Y \to S$, which is the restriction of $W$ along the continuous map $y \mapsto (x,y)$. In terms of open subsets, this amounts to starting with an open set $W$ in $X \times Y$ and pulling it back along the inclusion of $Y$ as a slice in $X \times Y$ to an open subset of $Y$.

\begin{deph}\label{defstrengthH}
	Let $X$ and $Y$ be topological spaces. We define $s : X \times HY \to H(X \times Y)$ on $x \in X$ and $C \in HY$ such that
	\begin{equation}
		\label{strength_def}
		\pair{s(x,C)}{W} \;\coloneqq\; \pair{C}{W(x,-)}
	\end{equation}
	for all $W \in S^{X \times Y}$.
\end{deph}

Since the restriction map $W \mapsto W_x$ preserves joins, the duality of \Cref{closedduality} applies, resulting in $s(x,C) \in H(X \times Y)$.
In terms of closed sets, we have
\[
	s(x,C) = \cl(\{x\} \times C ) = \cl(\{x\}) \times C,
\]
since $\{x\} \times C$ is disjoint from an open subset $W \subseteq X \times Y$ if and only if $C$ is disjoint from the slice of $W$ at $x$, and the second equation holds since products of closed sets are closed.

If $W = U \times V$ is a product of open sets, which means $W(x,y) = U(x) V(y)$ in terms of functions, then definition~\eqref{strength_def} simplifies to
\[
	\pair{s(x,C)}{U\times V} \;=\; U(x) \, \pair{C}{V} .
\]
The continuity of $s$ follows. \Cref{H_basis} and the fact that the products of open sets form a basis imply that it is enough to prove the continuity of the map
\[
	(x,C) \mapsto U(x) \, \pair{C}{V},
\]
which is indeed continuous since the multiplication $S \times S \to S$ is continuous and each factor of the product is.

\begin{prop}\label{H_strength_natural}
	The map $s$ is natural in both arguments: for all continuous functions $f:X\to Z$ and $g:Y\to W$, the following two diagrams commute.
$$
\begin{tikzcd}
 X\times HY \ar{d}{f\times \id} \ar{r}{s} & H(X\times Y) \ar{d}{(f\times \id)_\sharp} \\
 Z\times HY \ar{r}{s} & H(Z\times Y)
\end{tikzcd}
\qquad\qquad
\begin{tikzcd}
 X\times HY \ar{d}{\id\times g_\sharp} \ar{r}{s} & H(X\times Y) \ar{d}{(\id\times g)_\sharp} \\
 X\times HW \ar{r}{s} & H(X\times W)
\end{tikzcd}
$$
\end{prop}

\begin{proof}
Since the products of open sets $U \times V$ form a basis for the product topology, by \Cref{H_basis} it is enough in either case to show that both diagrams commute after pairing with a generic product of open sets $U \times V$. Doing so for the first diagram gives, for $x\in X$ and $C\in HY$,
	\begin{align*}
		\pair{(f\times \id)_\sharp s(x,C)}{U\times V} \;&=\; \pair{s(x,C)}{(U \times V) \circ (f \times \id)} \;=\; \pair{s(x,C)}{(U \circ f)\times V} \\
		&=\; U(f(x)) \, \pair{C}{V} \;=\; \pair{s(f(x),C)}{U\times V} .
	\end{align*}	
Similarly, for the second diagram,
	\begin{align*}
		\pair{(\id\times g)_\sharp s(x,C)}{U\times V} \;&=\; \pair{s(x,C)}{(U\times V) \circ (\id \times g)} \;=\; \pair{s(x,C)}{U \times (V \circ g)} \\
		&=\; U(x) \, \pair{C}{V \circ g} \;=\; U(x) \, \pair{g_\sharp C}{V} \;=\; \pair{s(x,g_\sharp C)}{U\times V}. \qedhere
	\end{align*}
\end{proof}

\begin{prop}
 \label{Hstrength}
 The map $s$ is a strength for the monad $H$.
\end{prop}

In other words, for all topological spaces $X$ and $Y$, the following four diagrams commute. The first two involve the unitor $u$ and associator $a$ of the Cartesian monoidal structure of $\cat{Top}$, while the other ones involve the structure maps of the monad.
$$
\begin{tikzcd}
 1 \times HX \ar{r}{s} \ar{dr}{\cong}[swap]{u} & H(1\times X) \ar{d}{\cong}[swap]{u_\sharp} \\
 & HX
\end{tikzcd}
$$

$$
\begin{tikzcd}
 (X\times Y) \times HZ \ar{rr}{s} \ar{d}{\cong}[swap]{a} && H((X\times Y)\times Z) \ar{d}{\cong}[swap]{a_\sharp } \\
 X\times(Y\times HZ) \ar{r}{\id\times s} & X\times H(Y\times Z) \ar{r}{s} & H(X\times(Y\times Z))
\end{tikzcd}
$$

$$
\begin{tikzcd}
 X\times Y \ar{r}{\id\times\sigma} \ar{dr}[swap]{\sigma} & X\times HY \ar{d}{s} \\
 & H(X\times Y)
\end{tikzcd}
$$

$$
\begin{tikzcd}
 X\times HHY \ar{r}{s} \ar{d}{\id\times\U} & H(X\times HY) \ar{r}{s_\sharp } & HH(X\times Y) \ar{d}{\U} \\
 X\times HY \ar{rr}{s} && H(X\times Y)
\end{tikzcd}
$$

\begin{proof}
	For the first diagram, let $C\in HX$ and $U \in S^X$. Denoting by $\bullet$ the unique point of $1$, we have
	\[
		\pair{u_\sharp (s(\bullet,C))}{U} \;=\; \pair{s(\bullet,C)}{U \circ u} \;=\; \pair{C}{U} \;=\; \pair{u(\bullet,C)}{U}.
	\]
	For the second diagram, we apply \Cref{H_basis} to reduce to the evaluation on $U \in S^X$, $V \in S^Y$, and $W \in S^Z$. For each $x\in X$, $y\in Y$, and $C\in HZ$, we have 
	\begin{align*}
		\pair{a_\sharp s((x,y),C)}{U\times (V\times W)} \;&=\; \pair{s((x,y),C)}{(U\times (V\times W)) \circ a} \\
		&=\; \pair{s((x,y),C)}{(U\times V)\times W} \\
		&=\; (U \times V)(x,y) \, \pair{C}{W} \\
		&=\; U(x) \, V(y) \, \pair{C}{W} \\
		&=\; U(x) \, \pair{s(y,C)}{V\times W} \\
		&=\; \pair{s(x,s(y,C))}{U\times (V\times W)} .
	\end{align*}
	For the third diagram, we similarly evaluate
	\begin{align*}
		\pair{s(x,\sigma(y))}{U\times V} \;&=\; U(x) \, \pair{\sigma(y)}{V} \;=\; U(x) \, V(y) \\
		&=\; (U \times V)(x,y) \;=\; \pair{\sigma((x,y))}{U\times V} .
	\end{align*}
	And for the last diagram,
	\begin{align*}
		\pair{\U(s_\sharp s(x,\mathcal{C}))}{U\times V} \;&=\; \pair{s_\sharp s(x,\mathcal{C})}{\pair{-}{U\times V}} \;=\; \pair{s(x,\mathcal{C})}{\pair{s(-)}{U\times V}} \\
		&=\; \pair{s(x,\mathcal{C})}{U\times \pair{-}{V}} \;=\; U(x) \, \pair{\mathcal{C}}{\pair{-}{V}} \\
		&=\; U(x) \, \pair{\U\mathcal{C}}{V} \;=\; \pair{s(x,\U\mathcal{C})}{U\times V} . \qedhere
	\end{align*}
\end{proof} 

\begin{prop}
 \label{H_commutative}
 The strength is commutative, in the sense that the following diagram commutes,
 $$
 \begin{tikzcd}
  HX \times HY \ar{d}{t} \ar{r}{s} & H(HX\times Y) \ar{r}{t_\sharp } & HH(X\times Y) \ar{d}{\U} \\
  H(X\times HY) \ar{r}{s_\sharp } & HH(X\times Y) \ar{r}{\U} & H(X\times Y)
 \end{tikzcd}
 $$
 where the costrength $t:HX\times Y\to H(X\times Y)$ is obtained from the strength via the braiding. Explicitly,
$$
	\pair{t(C,y)}{W} \; = \; \pair{C}{W(-,y)}
$$
for all $C\in HX$, $y\in Y$ and $W \in S^{X \times Y}$.
\end{prop}

\begin{proof}
Let $C\in HX$, $D\in HY$, $U \in S^X$, and $V \in S^Y$. Then
	\begin{align*}
		\pair{\U(t_\sharp s(C,D))}{U\times V} \;&=\; \pair{t_\sharp s(C,D)}{\pair{-}{U\times V}} \;=\; \pair{s(C,D)}{\pair{t(-)}{U\times V}} \\
		&=\; \pair{s(C,D)}{\pair{-}{U} \times V} \;=\; \pair{C}{U} \, \pair{D}{V}.
	\end{align*}
	An analogous computation shows that $\pair{\U(s_\sharp t(C,D))}{U\times V} = \pair{C}{U} \, \pair{D}{V}$ as well. This is enough because products of open sets form a basis.
\end{proof}

\noindent Per \Cref{comm_monads}, we have the following.

\begin{cor}
	$(H,\sigma,\U)$ is a commutative (equivalently symmetric monoidal) monad with respect to the strength $s$.
\end{cor}

The lax monoidal structure is implemented by the multiplication map $HX\times HY\to H(X\times Y)$ given by the product of closed sets $(C,D) \mapsto C\times D$, as the computation in the proof of \Cref{H_commutative} shows. Due to the universal property of the product in $\cat{Top}$, $H$ is an oplax monoidal monad as well, even bilax (see \Cref{comm_monads}), which underlines the analogy between $H$ and probability monads~\cite{ours_bimonoidal}. The comultiplication $H(X\times Y)\to HX\times HY$ projects a closed set in the product space $X\times Y$ to the pair of its projections to $X$ and $Y$.

\section{The monad $V$ of continuous valuations}
\label{sec_valuations}

A valuation is similar to a Borel measure, but is defined only on the open sets of a topological space (see, for example,~\cite{AJK}). Valuations appeared as generalizations of Borel measures better suited to the demands of point-free topology and constructive mathematics. Jones and Plotkin \cite{jones-plotkin} defined a monad of subprobability valuations on the category of directed complete partially ordered sets (dcpo's) and Scott-continuous maps. The underlying endofunctor of this monad assigns to a dcpo the set of Scott-continuous subprobability valuations, its \emph{probabilistic powerdomain}. The monad multiplication corresponds to forming the expected valuation by integration.
Kirch~\cite{kirch} generalized the construction by working with valuations taking values in $\Rplusext$, and obtained a monad on the category of continuous domains and Scott-continuous maps~\cite[Satz~6.1]{kirch}. He also proved a Markov-Riesz-type duality for valuations~\cite[Satz~8.1]{kirch}, namely, that on a core-compact space $X$ (for example, a continuous domain) there is a duality of cones between lower semicontinuous functions $X \to \Rplusext$ and continuous valuations.

Heckmann~\cite{heckmann95}\footnote{We refer to the preprint version \cite{heckmann95} throughout, which is freely available online. Although the paper is published~\cite{heckmann}, we have not been able to obtain the published version. In addition, some results about the monoidal structure of $V$ seem to be present only in \cite{heckmann95}.} constructed for every topological space $X$ a space $VX$ of continuous valuations with values in $\Rplusext$, and proved that the construction forms a monad~\cite[Section~10]{heckmann95} which was later~\cite{stablycompext} named \emph{extended probabilistic powerdomain}. Heckmann~\cite[Theorem~9.1]{heckmann95} also extended the duality result of Kirch, showing that on every topological space $X$ there is a bijection between continuous valuations on $X$ and Isbell-continuous linear functionals from lower semicontinuous functions $X\to\Rplusext$ to $\Rplusext$.
Alvarez-Manilla, Jung, and Keimel~\cite[Theorem~25]{AJK} showed that one can view the space of continuous valuations $VX$ as the space of Scott-continuous, monotone, linear functionals from the space of lower semicontinuous functions $X\to\Rplusext$ back to $\Rplusext$. This duality formula, which is analogous to the one for the monad $H$ (\Cref{closedduality}), will form the basis of our treatment of $V$. Vickers~\cite{monad-locales} and Coquand and Spitters~\cite{coquand-spitters} have generalized the construction and the duality result to locales. 
For a more detailed history of the continuous valuations monad on $\cat{Top}$, see the papers of Alvarez-Manilla et al.~\cite{AJK} and Goubault-Larrecq and Jia \cite{algebras}.

Very recently, and independently of us, Goubault-Larrecq and Jia \cite{algebras} have studied the algebras of $V$. They show that every $V$-algebra is a $T_0$ space endowed with a certain structure, called a weakly locally convex sober topological cone.
They also prove that under additional assumptions (such as core-compactness of the space), this structure is also sufficient to induce a $V$-algebra structure.
A full characterization of the algebras of such monads, and a general answer to the question of which cones are algebras, is at present lacking (and presumably quite difficult).
We now review the basic constructions and results pertaining to the monad $V$.

\begin{deph}[Continuous valuation]
\label{defn_cpv}
Let $X$ be a topological space. A continuous valuation on $X$ is a map $\nu: \mathcal{O} (X) \to \Rplusext$ that satisfies the following four conditions.
\begin{enumerate}
\item \emph{Strictness}: $\nu(\varnothing)=0$.
\item \emph{Modularity}: For any $U , V \in \op(X)$, we have
	\begin{equation}
		\label{modularity}
		\nu (U \cup V) + \nu (U \cap V) = \nu(U) + \nu(V).
	\end{equation}
\item \emph{Scott continuity}: For any directed net $(U_\lambda)_{\lambda \in \Lambda}$ in $\op(X)$, we have
$$
\nu \left( \bigcup_{\lambda \in \Lambda}  U_\lambda \right) = \bigvee_{\lambda \in A} \nu \left( U_\lambda \right).
$$
\end{enumerate}
\end{deph}

Since Scott continuity implies monotonicity, $\nu(U) \le \nu(V)$ for $U \subseteq V$, we do not list it as a separate condition. It would have to be included in the definition of not necessarily continuous valuations.

Below we consider the \emph{space} of continuous valuations $VX$ and the extended probabilistic powerdomain monad $V : \cat{Top} \to \cat{Top}$. In \Cref{valmeas} we also recall the connection between continuous valuations and Borel measures.

\begin{remark}
It is clear from the definition that the set of continuous valuations only depends on the frame of open subsets of $X$, that is on the sobrification of $X$. In particular, it only depends on the universal $T_0$ quotient of $X$, its Kolmogorov quotient.
\end{remark}

\subsection{Duality theory for continuous valuations}\label{ssecddual2}

One can define an integration theory for continuous valuations that is analogous to Lebesgue integration for measures, but the role of measurable functions is played by lower semicontinuous functions~\cite{kirch,jung}. This results in a duality between continuous valuations and lower semicontinuous functions, just as closed subsets are dual to open subsets in \Cref{closedduality}. 

We start by recalling the definition of integral of a lower semicontinuous function against a continuous valuation, also known as \emph{lower integral}.
As far as we know, it was first defined in Kirch's thesis~\cite{kirch}, written in German. A reference in English is the slightly later work of Heckmann~\cite{heckmann95}.

Throughout, we equip the extended nonnegative reals $\Rplusext$ with the \emph{upper topology}, whose open sets are the sets of the form $(r,\infty]$, in addition to $\varnothing$ and the whole space. This topology is also known as the \emph{topology of lower semicontinuity}, since continuous functions $f : X \to \Rplusext$ are characterized by the openness of the sets $f^{-1}((r,\infty])$, a property more commonly known as lower semicontinuity. As we will see, $\Rplusext$ plays a role for continuous valuations that is perfectly analogous to the one of the Sierpiński space throughout \Cref{hyperspace}.

For every space $X$, we denote the set of (lower semi-)continuous functions $X \to \Rplusext$ by $\Rplusext^X$. It is an important fact that all joins in $\Rplusext^X$ exist and are pointwise. As for $S^X$ in \Cref{hyperspace}, we equip $\Rplusext^X$ with the pointwise order and pointwise algebraic structure given by addition and multiplication, using the convention that $\infty \cdot 0 = 0 \cdot \infty = 0$. There is also an action of every $r \in \Rplusext$ by ``scalar multiplication'', corresponding to multiplication by the constant function $X \to \{r\}$.

As in Lebesgue integration theory, a lower semicontinuous function $X \to \Rplusext$ is called \emph{simple} if it assumes only finitely many values. Every simple lower semicontinuous function $f:X\to\Rplusext$ can be written as a positive linear combination of indicator functions of open sets, that is in the form
 $$
 f \;=\;\sum_{i=1}^n r_i \1_{U_i}
 $$
 for $r_i\in(0,\infty]$ and $U_i\subseteq X$ open for all $i$, as can be seen by induction on the number of values that $f$ takes.
 It is well known that every lower semicontinuous function $X\to \Rplusext$ can be expressed as a directed supremum of simple functions (see for example \cite[Lemma~1.2]{kirch}). 
The integral is defined such that it is continuous with respect to directed suprema.
\begin{deph}
	Let $X$ be a topological space and $\nu : \op(X) \to \Rplusext$ a continuous valuation. For a simple function
$$
f\;=\;  \sum_{i=1}^n r_i \1_{U_i},
$$
the integral of $f$ with respect to $\nu$ is given by
$$
\int f \, d\nu \;:=\;  \sum_{i=1}^n r_i\, \nu(U_i).
$$
For any $g\in \Rplusext^X$, the integral of $g$ with respect to $\nu$ is defined as
$$
\int g \, d\nu \;:=\; \sup \left\{ \int f \, d\nu \;\bigg|\; \text{$f$ is simple and $f\le g$} \right\} .
$$
\end{deph}
It is not immediately obvious that the integral of a simple function is well-defined, since a priori it may depend on the particular representation, but it does not \cite[Proposition~1.2 and Lemma~4.1]{kirch}.

\begin{nota}
	To emphasize the analogy between the integral and the pairing~\eqref{H_pairing} in the Hoare hyperspace case, we also write
	\[
		\pair{\nu}{g} \coloneqq \int g \, d\nu.
	\]
\end{nota}

This pairing notation is also motivated by the following known duality result, which can be seen as a counterpart of \Cref{closedduality}.
Here and in the following, ``linear'' and ``linearity'' refers to linear combinations with coefficients in $[0,\infty]$.

\begin{thm}[Representation theorem for valuations]\label{daniellval}
	For any topological space $X$, integration establishes a bijection between:
	\begin{enumerate}
		\item continuous valuations on $X$;
		\item Maps $\phi : \Rplusext^X \to \Rplusext$ such that the following two conditions hold:
			\begin{enumerate}
				\item \label{V_linearity} \emph{Linearity:} For every $f,g \in \Rplusext^X$, we have
					\[
						\phi(f + g) = \phi(f) + \phi(g),
					\]
					and for every $r \in \Rplusext$, we have $\phi(rf) = r \phi(f)$.
				\item \label{V_scottcontinuity} \emph{Scott continuity:} For any directed net $(f_\lambda)_{\lambda \in \Lambda}$ in $S^X$, we have
					\[
						\phi \left( \sup_{\lambda \in \Lambda} f_\lambda \right) = \sup_{\lambda \in \Lambda} \phi(f_\lambda).
					\]
			\end{enumerate}
	\end{enumerate}
	Under this bijection, the pointwise order on continuous valuations corresponds exactly to the pointwise order on maps $\Rplusext^X \to \Rplusext$.
\end{thm}

Similarly as for \Cref{closedduality}, ``linearity'' is intended between semimodules over the semiring $\Rplusext$ (without negatives).

Note that the Scott continuity of integration can be thought of as an incarnation of the monotone convergence theorem for continuous valuations.

The proof of \Cref{daniellval} is fairly straightforward: restricting a given $\phi$ to indicator functions of open sets produces a continuous valuation, and one only needs to argue that this construction is inverse to the formation of the integral. While this is obvious in one direction, the other follows since~\ref{V_linearity} and~\ref{V_scottcontinuity} guarantee that $\phi$ is uniquely determined by its values on indicator functions of open sets. This line of argument is used, for example, by Kirch~\cite[Satz~8.1]{kirch} for the case of core-compact $X$, where it is worth noting that Kirch's proof actually works for any topological space $X$.\footnote{The reason why Kirch states core-compactness as an assumption is that in his work linearity and Scott continuity of the integral are proven only for the core-compact case \cite[Satz~4.1 and Satz~4.2]{kirch}, although they do hold in general~\cite{heckmann95}.} A very similar duality result appears in Heckmann's work~\cite[Theorem~9.1]{heckmann95}.

Intuitively, a continuous valuation is a quantitative analogue of a closed set: a closed set may or may not hit a given open set, while a valuation assigns a numerical value to every open set. Closed subsets of $X$ are equivalent to particular maps $S^X \to X$, and the lower Vietoris topology corresponds to the topology of pointwise convergence of such functionals. Similarly, continuous valuations on $X$ are particular maps $\Rplusext^X \to \Rplusext$ by \Cref{daniellval}. As for $H$, this double dualization picture will turn out to be useful for the construction of the monad and especially for the treatment of its multiplication.

\begin{deph}[Space of continuous valuations]\label{val_topology}
	Let $X$ be a topological space.
	\begin{enumerate}
		\item The space of continuous valuations over $X$, denoted $VX$, is the set of continuous valuations on $X$, or, equivalently, the set of maps $\Rplusext^X \to \Rplusext$ of \Cref{daniellval}.
		\item We equip $VX$ with the weakest topology that makes the pairing maps
			\[
				\pair{-}{f} \: : \: HX \longrightarrow \Rplusext
			\]
			continuous for every $f \in \Rplusext^X$.
	\end{enumerate}
\end{deph}

\noindent This weak topology on $VX$ has a well-known analogue for Borel measures (see \Cref{atop}).

\noindent We now develop an analogue of \Cref{H_basis}, and consider the specialization preorder on $VX$.

\begin{lemma}\label{V_basis}
	Let $\mathcal{F} \subseteq \Rplusext^X$ be a subset such that every function $f \in \Rplusext^X$ is a directed supremum of linear combinations of functions in $\mathcal{F}$. Then:
	\begin{enumerate}
		\item\label{V_basis_specialization} For given $\nu,\rho \in VX$, the following are equivalent:
			\begin{enumerate}
				\item\label{V_F_pointwise} For all $f \in \mathcal{F}$, we have $\pair{\nu}{f} \le \pair{\rho}{f}$.
				\item\label{V_specialization} $\nu \le \rho$ in the specialization preorder on $VX$.
			\end{enumerate}
		\item The topology on $VX$ is the weakest topology which makes all of the pairings
			\[
				\pair{-}{f} \: : \: VX \longrightarrow S
			\]
			continuous for $f \in \mathcal{F}$.
	\end{enumerate}
\end{lemma}

The specialization order is also the canonical order on the space of valuations in the probabilistic powerdomain~\cite{jones-plotkin}.

\begin{proof}
The proof is analogous to the proof of \Cref{H_basis}.
\begin{enumerate}
\item The pairing map is linear and Scott-continuous in the second argument, and therefore condition~\ref{V_F_pointwise} holds if and only if it holds with $\Rplusext^X$ in place of $\mathcal{F}$. Having replaced $\mathcal{F}$ by $\Rplusext^X$, it is also clear that~\ref{V_F_pointwise} is equivalent to~\ref{V_specialization}, since~\ref{V_F_pointwise} states exactly that every neighborhood of $\nu$ is a neighborhood of $\rho$, which is the definition of the specialization preorder.
\item We need to show that, if these maps are continuous, then so is every $\pair{-}{f}$ for $f \in \Rplusext^X$. This follows since linear combinations of continuous maps to $\Rplusext$ are continuous, as are directed suprema. \qedhere
\end{enumerate}
\end{proof}

An interesting choice for $\mathcal{F}$ is given by the collection of indicator functions of open sets, for then the pairing maps $\pair{-}{\1_U}$ coincide with the evaluation maps $\nu \mapsto \nu(U)$, and the fact that every function in $\Rplusext^X$ is a directed supremum of simple functions implies that these evaluation maps generate the topology as well. In fact, we have the following stronger result of Heckmann, which is one of the few statements whose proof is not analogous to an argument from~\Cref{hyperspace}.

\begin{prop}[Proposition~3.2 in \cite{heckmann95}]
	\label{portmanteauval}
	Let $\mathcal{B} \subseteq \op(X)$ be a basis that is closed under finite intersections.
	Then the topology of $VX$ is characterized as making the evaluation maps $VX \longrightarrow \Rplusext$ assigning $\nu \longmapsto \nu(U)$ continuous for all $U \in \mathcal{B}$.
\end{prop}

\begin{proof}
Consider the weakest topology on $VX$ that makes the $\nu \mapsto \nu(U)$ continuous for all $U \in \mathcal{B}$. It is then enough to show that for every finite sequence $U_1, \ldots, U_n \in \mathcal{B}$, also the evaluation map $\nu \mapsto \nu(U_1 \cup \ldots \cup U_n)$ is continuous, since the continuity of evaluation on any open set then follows by Scott continuity. By induction, we can extend the modularity equation~\eqref{modularity} to the inclusion-exclusion formula
	\[
		\nu(U_1 \cup \ldots \cup U_n) = \sum_{k=1}^n (-1)^{k+1} \sum_{i_1 \le \ldots \le i_k} \nu(U_{i_1} \cap \ldots \cap U_{i_k}).
	\]
	This implies the claim since every term on the right-hand side is a continuous function of $\nu$ by the assumption that $\mathcal{B}$ is closed under finite intersections.
\end{proof}

\begin{prop}[Proposition~5.1 in \cite{heckmann95}]\label{VX_sober}	
For every topological space $X$, the space $VX$ is sober.
\end{prop}

\begin{proof}
Completely parallel to the proof of \Cref{HX_sober}.
\end{proof}

We end this subsection with a small excursion to ordered topological spaces and the \emph{stochastic order}, which plays an important role in applied probability and economic theory~\cite{stochastic-orders}. This is of independent interest and will play no further role in this paper. It illustrates the utility of working with non-Hausdorff spaces: as we will see, the stochastic order for continuous valuations on a preordered space (say a Hausdorff space) coincides with the above specialization order on a suitable non-Hausdorff space.

A \emph{preordered topological space} is a topological space $X$ which is also a preordered set $(X,\le)$ such that the set of all ordered pairs $\{(x,y) \mid x \le y\}$ is closed in the product topology of $X \times X$.

\begin{deph}[Stochastic order]
Let $X$ be a preordered topological space. For any two $\nu,\rho\in VX$, we put $\nu\le\rho$ if and only if $\nu(U) \leq \rho(U)$ for any open upper set $U \subseteq X$. 
\end{deph}

The stochastic order coincides with the specialization preorder of continuous valuations if one replaces the topology by only those opens that are \emph{upward closed} with respect to the preorder. This may be a useful thing to do since the pointwise order is conceptually simpler than the stochastic order. Since the topology consisting of the upward closed opens is typically non-Hausdorff, we find that continuous valuations on non-Hausdorff spaces occur naturally in the context of the stochastic order.

\begin{eg}[Stochastic dominance on the real line]
	The stochastic order on $\R$, considered as a preordered topological space with the standard topology and order, is widely used in decision theory, economics, and finance to compare probability measures on the real line (e.g.~\cite{rothschild70}). For two probability measures on $\R$, one has $p \leq q$ if and only if $p((a,\infty)) \le q((a,\infty))$ for all $a \in \R$.
	If we now consider the valuations obtained by restricting the probability measures to just the open set, the induced order is exactly the pointwise order for continuous valuations on $\R$ equipped with the upper topology.
\end{eg}

\subsection{Functoriality}
\label{val_functoriality}

We now show that $V$ is a functor $\cat{Top} \to \cat{Top}$, based on the double dualization approach,
as we had done for the Hoare hyperspace monad $H$ in \Cref{H_fun}.

\begin{deph}[Pushforward]
\label{V_functoriality}
	Let $f:X\to Y$ be continuous and $\nu\in VX$. The \emph{pushforward} $f_*\nu \in VY$ is defined through the pairing with any $g \in \Rplusext^Y$ by 
 $$
 \pair{f_*\nu}{g} \; = \; \pair{\nu}{g\circ f} .
 $$
\end{deph}

It is easy to see that $\pair{f_*(\nu)}{-}$ satisfies the conditions of \Cref{daniellval}: for any $g,h\in \Rplusext^Y$ and $r \in \Rplusext$ we have, since linear combinations of functions are pointwise,
\begin{align*}
&\pair{\nu}{(g + h)\circ f} = \pair{\nu}{(g\circ f) + (h\circ f)} = \pair{\nu}{g\circ f} + \pair{\nu}{h\circ f} , \\
&\text{and } \pair{\nu}{(rg) \circ f}  = \pair{\nu}{r (g \circ f)} = r \pair{\nu}{g \circ f} .
\end{align*}
For Scott continuity, consider a directed net $(g_\lambda)_{\lambda\in \Lambda}$ in $\Rplusext^X$. Then
\begin{align*}
	\pair{\nu}{\Big(\sup_{\lambda\in \Lambda} g_\lambda\Big)\circ f} = \pair{\nu}{\sup_{\lambda\in \Lambda} \, (g_\lambda\circ f)} = \sup_{\lambda\in \Lambda} \, \pair{\nu}{ g_\lambda\circ f} ,
\end{align*}
where the second step uses that directed suprema in $\Rplusext^X$ are pointwise.

We have thus turned the change of variables formula into the definition of the pushforward valuation, which also makes it immediate that $f_* : VX \to VY$ is continuous.
Plugging in the indicator function of an open set $U \subseteq Y$ for $g$ shows that we could equivalently have defined
 $$
 (f_*\nu)(U)\; = \;\nu(f^{-1}(U)).
 $$

\begin{lemma}
 Let $f:X\to Y$ and $g:Y\to Z$ be continuous maps between topological spaces. Then $(g\circ f)_* = g_* \circ f_*$.
\end{lemma}

The proof is exactly as in~\Cref{H_is_functor}. Since $(\id_X)* = \id_{VX}$ holds trivially, we have a functor $V:\cat{Top}\to\cat{Top}$ which assigns to a space $X$ its space of continuous valuations $VX$ and to each continuous map $f:X\to Y$ the pushforward map $f_*:VX\to VY$. In \Cref{probability}, it will turn out to be relevant that $V$ preserves subspace embeddings.

\begin{lemma}
	Let $i : X \hookrightarrow Y$ a homeomorphism onto its image. Then $i_* : VX \hookrightarrow VY$ is a homeomorphism onto its image as well. 
	\label{V_subspaces}
\end{lemma}

\begin{proof}
	By \Cref{portmanteauval}, it is enough to show that the induced map $\op(Y) \to \op(X)$ is surjective, which holds by assumption.
\end{proof}

\noindent The analogue of \Cref{H2monad} also holds, with essentially the same proof.

\begin{lemma}\label{V2monad}
	Let $X$ and $Y$ be topological spaces and let $f,g:X\to Y$ be continuous maps with $ f\le g$. Then $f_* \le g_*$. In other words, $V$ preserves 2-cells, making it into a 2-functor. 
\end{lemma}

\begin{proof}
	By \Cref{ordopen}, $f \le g$ implies that $h \circ f \le h \circ g$ in $\Rplusext^X$ for every $h \in \Rplusext^Y$. Therefore
  	\[
		\pair{f_*\nu}{h} = \pair{\nu}{h \circ f} \le \pair{\nu}{h \circ g} = \pair{g_*\nu}{h},
  	\]
	which, by \Cref{2cell_def} and \Cref{V_basis}, means $f_* \le g_*$.
\end{proof}

\subsection{Monad structure}\label{ssecmonadV}

As we did for $H$ in \Cref{hyperspace}, we equip $V$ with a monad structure using double dualization.

\subsubsection{Unit}
\label{V_unit}

\begin{deph}
\label{deltaval}
 Let $X$ be a topological space. The map $\delta:X\to VX$ is defined in terms of the pairing by
 \[
	\pair{\delta(x)}{f} \;\coloneqq\; f(x)
 \]
 for every $f \in \Rplusext^X$.
\end{deph}

It is obvious that linearity and Scott continuity in the second argument indeed hold, again since (directed) suprema in $\Rplusext^X$ are pointwise. $\delta$ is continuous by definition, since every $x \mapsto \pair{\delta(x)}{f}$ is continuous as a map $X \to \Rplusext$. We turn to showing that $\delta$ is a natural transformation $\delta : \id \Rightarrow V$ on $\cat{Top}$.

\begin{lemma}
 Let $f: X \to Y$ be continuous. Then the following diagram commutes.
 \[
  \begin{tikzcd}
   X \ar{r}{f} \ar{d}{\delta} & Y \ar{d}{\delta} \\
   VX \ar{r}{f_*} & VY
  \end{tikzcd}
 \]
\end{lemma}

\begin{proof}
 For $x\in X$ and $g \in \Rplusext^Y$,
$$
  \pair{f_*(\delta(x))}{g} \;=\; \pair{\delta(x)}{g\circ f} \;=\; g(f(x)) \;=\; \pair{\delta(f(x))}{g} .
  \qedhere
$$
\end{proof}

Using arguments analogous to the proof of \Cref{imagesigma2}, it can be shown that the image $\delta(X) \subseteq VX$, when equipped with the subspace topology, is homeomorphic to the Kolmogorov quotient of $X$ with respect to $\delta$ as the quotient map. In particular, $\delta$ is injective if and only if $X$ is $T_0$.

\subsubsection{Multiplication}

For every $g \in \Rplusext^X$, pairing with $g$ can be considered a function $\pair{-}{g} \in \Rplusext^{VX}$, which, as in \Cref{defmultH}, is central in the definition of the multiplication in terms of double dualization.

\begin{deph}
 Let $X$ be a topological space. The map $\E:VVX\to VX$ is defined on any $\xi\in VVX$ by
 $$ 
 \pair{\E\xi}{g} \;\coloneqq\; \pair{\xi}{\pair{-}{g}}
 $$
 for every $g \in \Rplusext^X$.
\end{deph}

Two applications of the linearity and Scott continuity of the pairing in the second argument show that this indeed results in an element $\E\xi$ of $VX$ by \Cref{daniellval}. It is also obvious that $\E : VVX \to VX$ is continuous. We now turn to naturality, which is proven in the same way as \Cref{unionnatural}.

\begin{prop}
 Let $f: X \to Y$ be continuous. Then the following diagram commutes.
 \begin{equation*}
  \begin{tikzcd}
   VVX \ar{r}{f_{**}} \ar{d}{\E} & VVY \ar{d}{\E} \\
   VX \ar{r}{f_*} & VY
  \end{tikzcd}
 \end{equation*}
\end{prop}
\begin{proof}
 Let $\xi\in VVX$ and $g \in \Rplusext^X$. Then
 \begin{align*}
 \pair{\E(f_{**}\xi)}{g} &= \pair{f_{**}\xi}{\pair{-}{g}} \;=\; \pair{\xi}{\pair{-}{g} \circ f_*} \;=\; \pair{\xi}{\pair{f_*(-)}{g}} \\
 &=\; \pair{\xi}{\pair{-}{g\circ f}} \;=\; \pair{\E\xi}{g\circ f} \;=\; \pair{f_*(\E\xi)}{g} . \qedhere
 \end{align*}
\end{proof}

\subsubsection{Monad axioms}

\begin{prop}
 Let $X$ be a topological space. Then the following three diagrams commute.
 \[
  \begin{tikzcd}
   VX \ar{r}{\delta} \idar{dr} & VVX \ar{d}{\E}  \\
   & VX
  \end{tikzcd}
  \qquad
  \begin{tikzcd}
   VX \ar{r}{\delta_*} \idar{dr} & VVX \ar{d}{\E} \\
   & VX
  \end{tikzcd}
  \qquad
  \begin{tikzcd}
   VVVX \ar{r}{\E_*} \ar{d}{\E } & VVX \ar{d}{\E} \\
   VVX \ar{r}{\E} & VX
  \end{tikzcd}
 \]
\end{prop}

\begin{proof}
The proof is the same as for \Cref{monaddiagramsH}, after the relevant exchange of symbols. We start with the first diagram, left unitality. For every $\nu\in VX$ and $g \in \Rplusext^X$,
$$
\pair{\E\delta(\nu)}{g} \;=\; \pair{\delta(\nu)}{\pair{-}{g}} \;=\; \pair{-}{g}(\nu) \;=\; \pair{\nu}{g} .
$$
Right unitality works similarly,
$$
\pair{\E\delta_*(\nu)}{g} \;=\; \pair{\delta_*\nu}{\pair{-}{g}} \;=\; \pair{\nu}{\pair{-}{g} \circ \delta} \;=\; \pair{\nu}{\pair{\delta(-)}{g}} \;=\; \pair{\nu}{g(-)} \;=\; \pair{\nu}{g} .
$$
It remains to consider the associativity diagram. For $\xi\in VVVX$ and $g \in \Rplusext^X$ we get
\begin{align*}
\pair{\E(\E_*\xi)}{g} \;&=\; \pair{\E_*\xi}{\pair{-}{g}} \;=\; \pair{\xi}{\pair{-}{g} \circ \E} \;=\; \pair{\xi}{\pair{\E(-)}{g}} \\
&=\; \pair{\xi}{\pair{-}{\pair{-}{g}}} \;=\; \pair{\E\xi}{\pair{-}{g}} \;=\; \pair{\E(\E\xi)}{g} . \qedhere
\end{align*}
\end{proof}

\noindent We have proven the following statement.

\begin{thm}
 The triple $(V,\delta,\E)$ is a monad on $\cat{Top}$. 
\end{thm}

We call $(V,\delta,\E)$, or just $V$, the \emph{monad of continuous valuations}. By \Cref{V2monad}, $V$ is a strict 2-monad for the 2-categorical structure of $\cat{Top}$ given in \Cref{2cat}.

\subsection{Algebras of $V$}\label{Valg}

Algebras of probability-like monads such as $V$ have some convex structure, in the sense that the algebra map can be interpreted as the formation of the barycenter of a measure or valuation. For example, the algebras of the Radon monad on the category of compact Hausdorff spaces are the compact convex subsets of locally convex topological vector spaces, with the algebra map given by barycenter formation~\cite{swirszcz,keimel}. The algebras of the Kantorovich monad on the category of complete metric spaces are the closed convex subsets of Banach spaces, again with respect to barycenter formation~\cite{ours_kantorovich}. If we drop the normalization of probability, as we do for $V$, then the algebra map $VA \to A$ can similarly be interpreted as assigning to every continuous valuation the integral of the identity function $A \to A$; in particular, such an algebra is a space where positive linear combinations of points can be evaluated to points. In contrast to the Radon and Kantorovich monads, a complete characterization of the category of algebras of $V$ seems to be quite difficult, but partial characterizations are possible. In work concurrent with ours, Goubault-Larrecq and Jia \cite{algebras} have proven that the algebras of $V$ on the subcategory of $T_0$ spaces are topological cones in the sense of Keimel~\cite{keimelcones} (as far as we know, the definition first appeared in~\cite[Section~7.1]{heckmann95}). Topological cones generalize convex cones in topological vector spaces without the requirement that addition is cancellative. We review Keimel's definition, and state some of the results of Goubault-Larrecq and Jia~\cite{algebras}. For further details we refer to the original papers~\cites{keimelcones, algebras}.

\begin{deph}[The category of topological cones]
 A topological cone is a $T_0$ topological space $K$ equipped with two operations.
 \begin{enumerate}
	\item An operation of addition $+:K\times K\to K$, jointly continuous, with a neutral element $0\in K$.
	\item An operation of scalar multiplication $\cdot: [0,\infty) \times K\to K$, jointly continuous (where $[0,\infty)$ carries the upper topology). 
 \end{enumerate}
  These operations satisfy the axioms of a $[0,\infty)$-semimodule with respect to the usual semiring structure of $[0,\infty)$. A morphism of topological cones is a continuous map which preserves the $[0,\infty)$-semimodule structure.
\end{deph}
 
\begin{prop}[Proposition~2 in \cite{stablycompext}]\label{goubault}
 Every $V$-algebra $e:VA\to A$ admits a canonical topological cone structure given by
 \begin{align*}
  a + b \;\coloneqq\; e(\delta_a + \delta_b) \hspace{20pt} \text{and} \hspace{20pt}
  r\,a \;\coloneqq\; e(r\,\delta_a) ,
 \end{align*} 
 where the sum and scalar multiplication on the right-hand sides are the pointwise sum and $r \in [0,\infty)$.
 Moreover, every morphism of $V$-algebras is a $[0,\infty)$-linear, continuous map.
\end{prop}

Note that this statement is analogous to the ``if'' part of \Cref{bothjoins}. Cohen et al.~\cite{stablycompext} proved this in the category of $T_0$ spaces. Since $VX$ is sober, and retracts of sober spaces are sober, every $V$-algebra is sober, which is the same reasoning as for $H$-algebras in~\Cref{Halg}. In particular every $V$-algebra is $T_0$, and therefore the proposition is true on the whole of $\cat{Top}$.
Moreover, Goubault-Larrecq and Jia~\cite[Proposition~4.9]{algebras} have shown that this cone is \emph{weakly locally convex}, meaning that for every point $x\in A$ and every open neighborhood $U\ni x$ there exists a convex neighborhood $C$ with $x\in C \subseteq U$ \cite[Definition~3.9]{algebras}.

Let $A$ be a topological cone. We say that $A$ is \emph{cancellative} if it is cancellative as a monoid, which means that for all $a,b,c\in A$,
$$
a + c \;=\; b + c \quad\Longrightarrow\quad a \;=\; b .
$$
Prominent examples of non-cancellative topological cones are topological complete join-semilattices (\Cref{top_lattice}).

\begin{eg}
Consider the lattice $W\coloneqq\{0,x,y,x\join y\}$ with the topological cone structure where the sum is given by the join, $0w= 0$, $\lambda w = w$ for all $w\in W$ and $\lambda>0$, and the open sets are the upper sets. This is a topological cone, and it is not cancellative. It is a topological complete join-semilattice in the sense of \Cref{top_lattice}, and therefore an $H$-algebra by \Cref{halgebras}. A systematic way in which lattices become $V$-algebras is given in \Cref{conseqalg}, where we show that every $H$-algebra is also canonically a $V$-algebra.
\end{eg}

\subsection{Products and marginals}
\label{V_bimonoidal}

In applications of measure theory, especially in probability theory, it is crucial to form products of measures and to project measures on product spaces to their marginal measures on the factor spaces. Even the fundamental probabilistic concept of stochastic independence can be understood in these terms: A product measure exhibits independence if it is equal to the product of its marginal measures. This subsection is concerned with the structure of products and marginals for $V$. We will show that $V$ is a commutative monad (see \Cref{comm_monads}), in the same way as we have done for $H$ (\Cref{H_bimonoidal}).

Instead of constructing products of valuations directly, it is easier to equip the monad with the equivalent structure of a \emph{commutative strength} (see \Cref{comm_monads}). This simpler approach is known to measure theorists--- if not under this name. For example, a map corresponding to the strength has been used by Ressel in his study of products of $\tau$-smooth Borel measures \cite{ressel77} (See our \Cref{ressel}). Conceptually, the use of the strength is as old as the concept of product measures; it is, for example, implicit in Halmos' treatment of product measures \cite[Paragraph 35]{halmos50}. The content of this section is mostly a unified treatment of results due to Heckmann \cite{heckmann95}.

To construct the strength transformation $s : X \times VY \to V(X \times Y)$, note that for every $g \in \Rplusext^{X \times Y}$ and $x \in X$, we have $g(x,-) \in \Rplusext^Y$. We have the following analogue of \Cref{defstrengthH}.

\begin{deph}[Proposition~11.1 in \cite{heckmann95}]
	Let $X$ and $Y$ be topological spaces. We define $s : X \times VY \to V(X \times Y)$ on $x\in X$ and $\nu\in VY$ such that
	\[
		\pair{s(x,\nu)}{f} \;\coloneqq\; \pair{\nu}{f(x,-)}
	\]
	for all $f \in \Rplusext^{X \times Y}$.
\end{deph}

The relevant properties for $s(x,\nu) \in V(X \times Y)$ are straightforward to check. If $f = gh$ is a product of $g \in \Rplusext^X$ and $h \in \Rplusext^Y$, then the definition simplifies to
\begin{equation}
	\label{strength_factor}
	\pair{s(x,\nu)}{gh} \;=\; \pair{\nu}{g(x) h(-)} \;=\; g(x) \, \pair{\nu}{h}.
\end{equation}
The continuity of $s$ follows. By \Cref{portmanteauval} it would even be enough to consider indicator functions of product of open sets, while even on a general $g$ and $h$ the right-hand side of~\eqref{strength_factor} is continuous in $x$ and $\nu$ since the multiplication $\Rplusext \times \Rplusext \to \Rplusext$ is. More generally, we will use the fact that a continuous valuation on a product space is uniquely determined by its pairings with product functions (for any number of factors in the product).

\begin{prop}
	\label{V_strength_natural}
	The map $s$ is natural in both arguments: for all continuous functions $f:X\to Z$ and $g:Y\to W$, the following two diagrams commute.
$$
\begin{tikzcd}
 X\times VY \ar{d}{f\times \id} \ar{r}{s} & V(X\times Y) \ar{d}{(f\times \id)_\sharp} \\
 Z\times VY \ar{r}{s} & V(Z\times Y)
\end{tikzcd}
\qquad\qquad
\begin{tikzcd}
 X\times VY \ar{d}{\id\times g_\sharp} \ar{r}{s} & V(X\times Y) \ar{d}{(\id\times g)_\sharp} \\
 X\times VW \ar{r}{s} & V(X\times W)
\end{tikzcd}
$$
\end{prop}

The proof is straightforward by using~\eqref{strength_factor} and perfectly analogous to the one of~\Cref{H_strength_natural}, hence we omit it.

\begin{prop}
 \label{Vstrength}
 The natural transformation $s$ is a strength for the monad $V$.
\end{prop}

In other words, for all topological spaces $X$ and $Y$, the following four diagrams commute, where $u$ is the unitor and $a$ the associator of the (Cartesian) monoidal structure of $\cat{Top}$.

$$
\begin{tikzcd}
 1 \times VX \ar{r}{s} \ar{dr}{\cong}[swap]{u} & V(1\times X) \ar{d}{\cong}[swap]{u_*} \\
 & VX
\end{tikzcd}
$$

$$
\begin{tikzcd}
 (X\times Y) \times VZ \ar{rr}{s} \ar{d}{\cong}[swap]{a} && V((X\times Y)\times Z) \ar{d}{\cong}[swap]{a_*} \\
 X\times(Y\times VZ) \ar{r}{\id\times s} & X\times V(Y\times Z) \ar{r}{s} & V(X\times(Y\times Z))
\end{tikzcd}
$$

$$
\begin{tikzcd}
 X\times Y \ar{r}{\id\times\delta} \ar{dr}[swap]{\delta} & X\times VY \ar{d}{s} \\
 & V(X\times Y)
\end{tikzcd}
$$

$$
\begin{tikzcd}
 X\times VVY \ar{r}{s} \ar{d}{\id\times\E} & V(X\times VY) \ar{r}{s_*} & VV(X\times Y) \ar{d}{\E} \\
 X\times VY \ar{rr}{s} && V(X\times Y)
\end{tikzcd}
$$

Since the statement and proof are perfectly analogous to the one of \Cref{Hstrength}, we omit the detailed verification here.

The following proposition defines the product valuations via the diagonal composite map of the diagram. Following Kock~\cite[Section~5]{kock_distributions}, this result can be thought of as a version of Fubini's theorem for continuous valuations (which is how Kock interprets the commutativity of a monad in general).

\begin{prop}
 \label{V_commutative}
 The strength of $V$ is commutative in the sense that the following diagram commutes,
 $$
 \begin{tikzcd}
  VX \times VY \ar{d}{t} \ar{r}{s} & V(VX\times Y) \ar{r}{t_*} & VV(X\times Y) \ar{d}{\E} \\
  V(X\times VY) \ar{r}{s_*} & VV(X\times Y) \ar{r}{\E} & V(X\times Y)
 \end{tikzcd}
 $$
 where the costrength $t:VX\times Y\to V(X\times Y)$ is obtained from the strength via the braiding. Explicitly,
$$
\pair{t(\nu,y)}{f} \;=\; \pair{\nu}{f(-,y)}
$$
for all $\nu\in VX$, $y\in Y$, $f\in\Rplusext^{X \times Y}$. 
\end{prop}

The proof of~\Cref{H_commutative} only needs to be copied and the corresponding symbols replaced.

\begin{cor}
	$(V,\delta,\E)$ is a commutative (equivalently symmetric monoidal) monad with respect to the strength $s$.
\end{cor}

In particular, we have a natural map $VX \times VY \to V(X \times Y)$ implementing the formation of products of continuous valuations~\cite[Theorem~11.2]{heckmann95}. The product of $\nu$ and $\rho$ pairs with functions of the form $g\cdot h$ as
$$
	\pair{\nu \otimes \rho}{g h} \;=\; \pair{\nu}{g} \, \pair{\rho}{h} .
$$
On open sets, the product valuation assigns
$$
	(\nu \otimes \rho)(U\times V) \;\longmapsto\; \nu(U)\cdot\rho(V) .
$$
Both formulations correspond to the ordinary product formulas widely used in measure and probability theory.

As for $H$, the universal property of the product makes $(V,\delta,\E)$ also into an oplax, and therefore even a bilax, monoidal monad. The resulting comultiplication $V(X\times Y)\to VX\times VY$ is the marginalization of valuations. The bilax monoidal structure of measure-like monads such as $V$ plays an important role in probability theory~\cite{ours_bimonoidal}.

\section{The probability monad on $\cat{Top}$}\label{probability}

The first probability functor was defined by Lawvere \cite{lawveremonad}: It assigns to a measurable space $X$ with $\sigma$-algebra $\Sigma_X$ the space of probability measures on $\Sigma_X$ endowed with the initial $\sigma$-algebra with respect to the family of evaluation maps $p \mapsto p (M)$ for $M \in \Sigma_X$.
This functor on the category of measurable spaces carries a canonical monad structure which makes it into a probability monad \cite{girymonad}.
However, due to the importance of topological concepts to measure theory and probability, such as weak convergence of measures (which requires a topology on the underlying space) or tightness, one can argue that there should be a probability monad on a suitable category of topological spaces as well, since only then can one hope to have enough expressiveness to develop synthetic versions of those theories in monadic terms.

The subcategory of topological spaces that is typically considered in analytical settings is the category of Polish spaces. Indeed Giry \cite{girymonad} introduced a probability monad on the category of Polish spaces: a Polish space $X$ is mapped to the space of Borel probability measures on $X$ equipped with the weak topology with respect to the integration maps $p \mapsto \int f \, dp$ for bounded continuous $f: X \to \R$, quite analogous to the topology on the space of continuous valuations that we had considered in the previous section.

Other choices of subcategories of $\cat{Top}$ are possible. One convenient subcategory for point-set topological purposes is the category of compact Hausdorff spaces. The Radon monad arising from the functor that assigns to a compact Hausdorff space the respective space of Borel probability measures equipped with the weak topology has been introduced by \'Swirszcz \cite{swirszcz}. A thorough treatment has been given by Fedorchuk \cite{fedorchuk}.

Neither of these monads restricts to the other, since the full subcategories of Polish spaces and of compact Hausdorff spaces overlap only partially in $\cat{Top}$.
But they are both contained in the category of $T_{3\frac{1}{2}}$ spaces. This case has been treated by Banakh \cite{banakh}, who studies the functor that assigns to a $T_{3\frac{1}{2}}$ space the space of inner regular and $\tau$-smooth Borel probability measures endowed with the weak topology, as well as the respective monad structure. 

In this section, we introduce the probability monad of $\tau$-smooth (but not necessarily inner regular) Borel probability measures on all of $\cat{Top}$. The underlying functor assigns to a topological space $X$ the space of $\tau$-smooth Borel probability measures on $X$ equipped with the A-topology, which coincides with the weak topology with respect to bounded continuous functions whenever $X$ is $T_{3\frac{1}{2}}$. Equipping this functor with a monad structure yields an extension of all the topological probability monads mentioned above. Since a probability measure on a Polish space is automatically $\tau$-smooth, our monad restricts to Giry's on Polish spaces. On compact Hausdorff spaces, the Radon probability measures are exactly the $\tau$-smooth ones, and we recover the Radon monad. On $T_{3\frac{1}{2}}$ spaces, Banakh's monad is a proper submonad of ours.

While our generalization is of independent interest, it is also motivated by applications to theoretical computer science: the study of probabilistic nondeterminism in denotational semantics requires the treatment of probability measures on topological spaces which may not even be $T_1$, let alone $T_{3\frac{1}{2}}$. Such spaces arise for example as partially ordered sets equipped with an order-compatible topology \cite{nonhausdorff}.

We finish this introduction with probability monads in metric settings. Breugel defined a probability monad on the category of compact metric spaces and 1-Lipschitz maps \cite{breugel}. The respective functor assigns to a compact metric space the space of its Borel probability measures endowed with the optimal transportation distance. Since the underlying space is compact, the optimal transportation distance induces the weak topology. This monad is a restriction of the probability monad of compact Hausdorff spaces of \'Swirszcz. However, the subcategory of topological spaces that is most convenient for geometrical purposes is the category of complete metric spaces and 1-Lipschitz maps. A probability monad on this category has been studied by Fritz and Perrone \cite{ours_kantorovich}, based on a categorical construction which does not involve measure theory. Its underlying functor assigns to a complete metric space the space of Radon probability measures with finite first moment, the largest subset of the Radon probability measures metrized by the optimal transportation distance.

\subsection{$\tau$-smooth Borel measures}\label{valmeas}

\begin{deph}[$\tau$-smooth Borel measure]
 Let $X$ be a topological space. A Borel measure $m$ on $X$ is called $\tau$-smooth if, for every directed net $(U_\lambda)_{\lambda\in\Lambda}$ of open subsets of $X$, 
 \begin{equation*}
  m \left( \bigcup_{\lambda\in\Lambda} U_\lambda \right) = \sup_\lambda m(U_\lambda) .
 \end{equation*}
\end{deph}

A more common regularity assumption on measures is the Radon property. It is known that, in the category of Hausdorff spaces, every Radon measure is $\tau$-smooth, and that, in the category of compact Hausdorff spaces, Radon measures and $\tau$-smooth measures coincide~\cite[Proposition~7.2.2]{bogachev}. Therefore all statements of this paper apply to Radon measures on Hausdorff spaces. It is clear that every $\tau$-smooth Borel measure restricts to a continuous valuation. 

\begin{deph}
	Let $\nu$ be a continuous valuation on a topological space $X$. We say that $\nu$ is \emph{extendable}, or that it \emph{extends to a measure}, if there exists a Borel measure $m$ such that, for every open set $U\subseteq X$, we have $m(U)=\nu(U)$.
\end{deph}

Since a Borel measure is uniquely determined by its values on open sets, such an extension, if it exists, is unique and $\tau$-smooth.

\begin{prop}\label{agreeonfunctions}
Let $X$ be a topological space. Suppose that a continuous valuation $\nu\in VX$ is extendable to a Borel measure $m$ on $X$. Then, for every lower semicontinuous function $f\in\Rplusext^X$, we have
$$
\pair{\nu}{f} \;=\;\int_X f \, dm.
$$
\end{prop}

Note that all functions in $\Rplusext^X$ are measurable and nonnegative, therefore their Lebesgue integral against a probability measure is either a well-defined nonnegative number or $+\infty$. To prove the statement, we use the following standard result. 

\begin{prop}[Corollary~414B.a in \cite{fremlin}]
\label{lsctausmooth}
 Let $X$ be a topological space and $\mu$ a $\tau$-smooth Borel measure on $X$. Let $(f_\lambda)_{\lambda\in\Lambda}$ be a directed net in $\Rplusext^X$. Define $f(x) := \sup_\lambda f_\lambda(x)$ for all $x\in X$. Then $f$ is lower semicontinuous and
 $$
  \int f \, d\mu \quad = \quad \sup_{\lambda} \int f_\lambda\,d\mu .
 $$
\end{prop}

\begin{proof}[Proof of~\Cref{agreeonfunctions}]
If $f$ is simple, the two quantities agree by definition of integral of a simple function (both for measures and for valuations), using the fact that $\nu$ and $m$ agree on open sets. 
If $f$ is not simple, we use the defining supremum of $\pair{\nu}{f}$ and \Cref{lsctausmooth} to obtain the desired equality.
\end{proof}

If $X$ is not sober, a continuous valuation need not be extendable, as the following counterexample illustrates. It is based on the fact that the $\{0,\infty \}$-valued continuous valuations with $\nu(X) = \infty$ correspond to completely prime filters on $\op(X)$.

\begin{eg}
\label{inextendable}
Let $X \coloneqq (0,1)$ carry the upper open topology, hence its open subsets are of the form $(a,1)$ for $a\in[0,1]$. The Borel $\sigma$-algebra of $X$ coincides with the Borel $\sigma$-algebra of $(0,1)$ in the Euclidean topology. Consider the continuous valuation $\nu: \op (X)\to \Rplusext$ given by
 $$
 \nu(U) \;\coloneqq\; \begin{cases}
	0 & \text{if } U=\varnothing, \\
 	1 & \text{otherwise}.
               \end{cases}
 $$
 Suppose that there was a Borel measure $m$ on $X$ which agrees with $\nu$ on the open sets. Consider the measurable set $(a,b]$ for $0<a<b<1$. We have
 $$
 m\big( (a,b] \big) \;=\; m\big( (a,1) \backslash (b,1) \big) \;=\; m\big( (a,1) \big) -  m\big( (b,1) \big) \;=\; 1-1 \;=\; 0 .
 $$
The space $X$ can be expressed as a countable disjoint union,
 $$
 X \;=\; (0,1) \;=\; \bigsqcup_{n=1}^\infty \left( 1-\dfrac{1}{n}\;,\; 1-\dfrac{1}{n+1} \right] .
 $$
 Note that $m(X)=\nu(X)=1$, while
 $$
 m(X) = \sum_{n=1}^{\infty} m \left( \left( 1-\dfrac{1}{n}\;,\; 1-\dfrac{1}{n+1} \right] \right) \;=\;  \sum_{n=1}^{\infty} 0 \;=\; 0.
 $$
 Therefore $m$ is not countably additive.
\end{eg}

The interpretation of the above example is that $\nu$ represents a Dirac mass at $1$, a point in the sobrification of $X$. If we sobrify $X$ to $(0,1]$ by including the point $1$, then $\nu$ can be extended to $\delta_1$.

It is known that, on a $T_{3\frac{1}{2}}$ space, every finite continuous valuation extends to a measure \cite{ext-val}. The same holds on spaces that are sober and locally compact \cite{ext-val-sober}. In particular, the sets of finite $\tau$-smooth Borel measures and of finite continuous valuations are in bijection for all locally compact Hausdorff spaces and for all metric spaces. Whether extensions exist for all (finite) continuous valuations on sober spaces seems to be an open question.

\subsection{The A-topology}\label{atop}

We will construct our monad $P$ as a submonad of the monad $V$ from \Cref{sec_valuations}.

\begin{deph}\label{Atopology}
Let $X$ be a topological space. We define the space $PX$ to be the set of $\tau$-smooth Borel probability measures on $X$, equipped with the subspace topology inherited from $VX$, which is the weakest topology which makes the integration maps
\[
	p \longmapsto \int f \, dp	
\]
continuous for every lower semicontinuous $f : X \to \Rplusext$.
\end{deph}

By \Cref{portmanteauval} and \Cref{agreeonfunctions}, we can equivalently say $PX$ carries the weakest topology which makes the evaluation maps $p \mapsto p(U)$ lower semicontinuous for all $U \in \op(X)$.
This topology is called the \emph{A-topology} \cite[8.10(iv)]{bogachev}, named after Pavel Alexandrov.
For a thorough study of the A-topology in the case of a Hausdorff space, consult the monographs by Tops{\o}e\footnote{Note that Tops{\o}e calls the A-topology the weak topology.} \cite[Part II]{topsoe} and Bogachev \cite[Section~8.10 (iv)]{bogachev}.

\begin{prop}[p.~227 of \cite{bogachev}]
	\label{weaktop}
	Let $X$ be a $T_{3\frac{1}{2}}$ space. Then the topology of $PX$ is equivalently the weakest topology which makes the integration maps
	\[
		p \longmapsto \int f \, dp
	\]
	continuous for all bounded continuous $f : X \to \R$.
\end{prop}

\begin{proof}
One direction is obvious, since to every bounded continuous $f : X \to \R$ one can add a constant so as to make it into a lower semicontinuous function with values in $\Rplusext$.

For the other direction, we equip $PX$ with the weak topology with respect to bounded continuous functions, for which we prove that the evaluation map $p \mapsto p(U)$ is lower semicontinuous for any $U \in \op(X)$. But since \Cref{agreeonfunctions} shows that the integral is Scott-continuous on $\Rplusext$-valued lower semicontinuous functions, it is enough to exhibit the indicator function $\1_U$ as a directed supremum of bounded continuous functions.

To do so, choose for every $x \in U$ a continuous function $f_x : X \to [0,1]$ with $f_x(x) = 1$ and $f_x|_{X \setminus U} = 0$, which is possible by the $T_{3\frac{1}{2}}$ assumption. For a finite subset $F \subseteq U$, consider $f_F := \max(f_1,\ldots,f_n)$. Then $\1_U$ is the directed supremum of the $f_F$ by construction, establishing the claim.
\end{proof}

The weak topology of \Cref{weaktop} is the one with respect to which convergence of measures is usually considered (as in the Portmanteau theorem). 
This is not the case for general topological spaces, for which the A-topology can be considered to be more well-behaved than the weak topology (see for example the discussion in \cite{bogachev} after Theorem~8.2.3 therein, as well as its Section~8.10.iv). This is why we work with the A-topology. Due to \Cref{weaktop}, all results stated for the A-topology apply also to the equivalent weak topology in the $T_{3\frac{1}{2}}$ case. This includes all metric spaces and all compact Hausdorff spaces. 

\subsection{Functoriality}

For topological spaces $X$ and $Y$, every continuous map $f : X \to Y$ is Borel measurable. Therefore, for $p \in PX$, we have the pushforward measure $f_*p \in PY$ defined by
\[
	(f_*p)(B) := p(f^{-1}(B))
\]
for all Borel sets $B \subseteq Y$. It is easy to see that $f_*(p)$ is also $\tau$-smooth. Suppose further that $p \in PX$ restricts to $\nu \in VX$. Then we have, for all $U \in \op(Y)$,
$$
(f_*\nu) (U) = \nu \left( f^{-1}(U) \right) = p \left( f^{-1}(U) \right) = (f_*p) (U) ,
$$
where the first equation holds by \Cref{V_functoriality}. Hence $f_*p$ restricts to $f_*\nu$. We have proven the following proposition.

\begin{prop}
$P$ is a subfunctor of $V$.
\end{prop}

We can further conclude from \Cref{V_subspaces} that, if $i : Y \hookrightarrow X$ is a subspace embedding, then so is $i_* : PY \hookrightarrow PX$.

\subsection{Monad structure}

Here we prove that $P$ can be extended to a submonad of $V$. Since $P$ is already a subfunctor of $V$, this monad structure is necessarily unique: we only need to show that both the unit and the multiplication of $V$ restrict to $P$.

Consider first the unit $\delta: X \to VX$. Since the Dirac valuation $\delta(x)$ defined by $\delta(x)(U) = \1_U(x)$ for open $U \subseteq X$ obviously extends to the Dirac measure $\delta_x (B) \coloneqq \1_B (x)$ for Borel sets $B \subseteq X$, the map $\delta$ factors through the inclusion $PX \hookrightarrow VX$. We also write $\delta : X \to PX$, slightly abusing notation. We know from \Cref{V_unit} that $\delta$ is a homeomorphism onto its image if $X$ is $T_0$. For $P$, the following more precise statement is known.

\begin{thm}[Theorem~11.1 of \cite{topsoe}]
	Let $X$ be a Hausdorff space. Then the map $\delta: X \to PX$ is the embedding of a closed subspace.
\end{thm}

The remaining issue in showing that $P$ is a submonad of $V$ is to prove that the multiplication $\E : VVX \to VX$ restricts to a map $\e : PPX \to PX$. To simplify the exposition, we construct $\e : PPX \to PX$ first and then show that $\e$ is indeed the restriction of $\E$.

\begin{prop}\label{evalmeasurable}
 Let $X$ be a topological space and let $A\subseteq X$ be a Borel set. Then the evaluation map $\eval_A:PX\to\R$ given by $p\mapsto p(A)$ is Borel measurable.
\end{prop}
\begin{proof}
	If $A\subseteq X$ is open, then $\eval_A$ is lower semicontinuous by definition of the A-topology on $PX$, and therefore also Borel measurable.
	
	In general, we denote by $\Sigma$ the set of Borel measurable $A\subseteq X$ for which $\eval_A:PX\to\R$ is measurable.
	To see that $\Sigma$ is closed under countable unions, suppose that a measurable set $B\subseteq X$ can be written as a countable union of disjoint measurable subsets,
 $$
 B = \bigcup_{n=1}^\infty A_n,
 $$
 such that $A_n\in \Sigma$ for every $n$. Then, for each $p\in PX$,
 $$
 \eval_B(p) \; = \; p(B) \; = \; \sum_{n=1}^\infty p(A_n) \; = \; \sum_{n=1}^\infty \eval_{A_n}(p) .
 $$
 Since pointwise suprema of measurable functions are measurable, and every partial sum on the right is measurable in $p$, we conclude that $\eval_A$ is also measurable in $p$. Therefore $\Sigma$ is closed under countable disjoint unions. Consider $A,B\in\Sigma$ with $A\subseteq B$. Clearly $\eval_{B\backslash A} = \eval_B - \eval_A$ is also measurable. Therefore $\Sigma$ is closed under relative complements. This makes $\Sigma$ into a Dynkin system. Since it contains the $\pi$-system of open subsets, the $\pi$-$\lambda$ theorem implies that $\Sigma$ is the $\sigma$-algebra of Borel sets.
\end{proof}

\begin{deph}
Let $X$ be a topological space and $\mu\in PPX$. Let $A\subseteq X$ be Borel measurable. We define
 $$
 (\e \mu) (A) \coloneqq \int_{PX} p(A) \, d\mu(p) .
 $$
\end{deph}

The integrand $p\mapsto p(A) \in [0,1]$ is measurable by \Cref{evalmeasurable} and bounded, therefore the integral exists.

\begin{prop}
	Let $X$ be a topological space and consider $\mu\in PPX$. Then the assignment $A\mapsto (\e \mu)(A)$ is a $\tau$-smooth probability measure on $X$. 
\end{prop}

\begin{proof}
 The nontrivial properties to establish are $\sigma$-additivity and $\tau$-smoothness. For the former, let $(A_n)_{n\in\N}$ be a countable family of disjoint measurable subsets of $X$ and let $A$ be their union. Then
 \begin{align*}
  (\e \mu)(A) &= \int_{PX} p(A) \, d\mu(p) = \int_{PX} \left( \sum_{n=1}^{\infty} p(A_n) \right) \, d\mu(p) \\
  &= \sum_{n=1}^\infty \int_{PX}   p(A_n) \, d\mu(p) = \sum_{n=1}^{\infty} \e \mu (A_n) ,
 \end{align*}
 where the third step can be thought of as an application of Fubini's theorem to $PX \times \N$.

 Turning to $\tau$-smoothness, let $(U_\lambda)_{\lambda\in\Lambda}$ be a directed net of open sets with union $U$. Then, since every measure $p\in PX$ is $\tau$-smooth,
\begin{align*}
 (\e \mu)(U) & = \int_{PX} p(U) \, d\mu(p) = \int_{PX} \sup_{\lambda} p(U_\lambda) \, d\mu(p) \\
 & = \sup_{\lambda} \int_{PX}  p(U_\lambda) \, d\mu(p) = \sup_{\lambda} \, (\e \mu)(U_\lambda),
\end{align*}
where the third step is an application of $\tau$-smoothness of $\mu$ (in the form of \Cref{agreeonfunctions} and Scott continuity of integration against a continuous valuation).
\end{proof}

Hence we have a well-defined map $\e: PPX \to PX$. The next proposition implies that this map is continuous, which we therefore do not prove separately.

The inclusion $\iota: PX \hookrightarrow VX$ is, by definition, a homeomorphism onto its image. Therefore so is $\iota_* : VPX \hookrightarrow VVX$ by \Cref{V_subspaces}. By composing these two embeddings, we can consider $PPX$ as a subspace of $VVX$.

\begin{prop}
\label{PPVV}
Let $X$ be a topological space. Then the following diagram commutes.
\[
	\begin{tikzcd}
		PPX \ar{d}{\e} \ar{r}{\iota}	& VPX \ar{r}{\iota_*}	& VVX \ar{d}{\E} \\
		PX \ar{rr}{\iota}		&			& VX
	\end{tikzcd}
\]
\end{prop}
\begin{proof}
	Let $\mu \in PPX$ and let $\nu \in VVX$ be its image. We show that $\E \nu$ extends to the measure $\e \mu$, by showing that they evaluate to the same number on any open set $U \subseteq X$,
\begin{align*}
	(\E \nu) (U) &= \pair{\E \nu}{\1_U} = \pair{\nu}{\pair{-}{\1_U}} = \int_{VX} \pair{\rho}{\1_U} \, d \mu(\rho) \\
	&= \int_{PX} \eval_U(p) \, d\mu(p)  = \int_{PX} p(U) \, d\mu(p)  = (\e \mu) (U). \qedhere
\end{align*}
\end{proof}

Since the unit and the multiplication of $V$ restrict to $P$, the equations required for $P$ to be a monad follow automatically from those of $V$. In particular, we do not need to prove the associativity of $E$. We have proven the following theorem.

\begin{thm}
	The triple $(P,\delta,E)$ is a submonad of $(V, \delta, \E)$ on $\cat{Top}$.
\end{thm}

\noindent Just as $V$, the monad $P$ is a strict 2-monad if we consider $\cat{Top}$ as a 2-category.

Through the submonad embedding, every algebra of $V$ is also an algebra of $P$ in a canonical way, resulting in a faithful functor from $V$-algebras to $P$-algebras. The question whether this forgetful functor is full or essentially surjective is related to the extension problem and is, at present, unanswered.

\subsection{Product and marginal probabilities}

Here, we show that $P$ is a commutative monad by equipping it with a commutative strength. This in particular implies that we have a natural transformation $PX \times PY \to P(X\times Y)$ implementing the formation of product probability measures. This is a non-trivial statement due to the discrepancy between the Borel $\sigma$-algebra on $X\times Y$ and the product of the Borel $\sigma$-algebras on $X$ and $Y$, which leads to known subtleties with the formation of product measures~\cite{borelproduct}.

As with the monad multiplication, this strength $X \times PY \to P(X \times Y)$ is inherited from the strength of $V$. While we only need to prove that the strength of $V$ restricts to $P$, it is more convenient to first construct the strength map $s : X \times PY \to P(X \times Y)$ explicitly, which we denote by the same symbol as for $V$.

\begin{lemma}\label{measslices}
	Let $X$ and $Y$ be topological spaces, and let $A\subseteq X\times Y$ be a Borel subset. Then, for each $x\in X$, the \emph{slice}
	\[
		A_x := \{ y \in Y \mid (x,y) \in A\}
	\]
	is a Borel subset of $Y$.
\end{lemma}

\begin{proof}
	The map $y \mapsto (x,y)$ is continuous, and hence Borel measurable, and $A_x$ is the preimage of $A$ under this map.
\end{proof}

For given $x \in X$ and $p \in PY$, we define the Borel probability measure $s(x,p)$ on $X \times Y$ on a Borel set $A\subseteq X\times Y$ as
$$
s(x,p)(A) \;\coloneqq\; p(A_x) .
$$
This is a $\tau$-smooth probability measure because taking the slice preserves arbitrary unions, intersections, and complements.

\begin{prop}
	Let $X$ and $Y$ be topological spaces. Then the following diagram commutes.
	\[
		\begin{tikzcd}
			X \times PY \ar{d}{s} \ar{r}{\id \times \iota}	& X \times VY \ar{d}{s} \\
			P(X \times Y) \ar{r}{\iota}			& V(X \times Y)
		\end{tikzcd}
	\]
\end{prop}

\begin{proof}
Suppose that $p\in PY$ and denote $\nu := \iota(p)$. On a product of open sets $U \times V \subseteq X \times Y$, we have
$$
s(x,p)(U \times V) \;=\; p((U \times V)_x) \;=\; \1_U(x) \, p(V) \;=\; \1_U(x) \, \nu(V) \;=\; s(x,\nu)(U \times V) ,
$$
 which is sufficient to compare two continuous valuations by \Cref{portmanteauval}.
\end{proof}

It is clear that $s : X \times PY \to P(X \times Y)$ inherits the naturality and strength properties from the strength of $V$. Thus we have shown the following.

\begin{cor}
The strength $s$ of $V$ restricts to a strength of the submonad $P$ which makes $P$ into a commutative monad.
\end{cor}

The commutativity of the strength of $P$ is Fubini's theorem~\cite{kock_distributions}. In particular, $P$ is also a symmetric monoidal monad (\Cref{comm_monads}), and the operation of forming product measures is a natural transformation $PX \times PY \to P(X \times Y)$. This recovers the following result of Ressel.

\begin{cor}[Theorem~1 in \cite{ressel77}]\label{ressel}
	The product of two $\tau$-smooth Borel measures on any two topological spaces $X$ and $Y$ extends to a $\tau$-smooth Borel measure on the product space $X \times Y$.
\end{cor}

Moreover, since $\iota$ is a morphism of commutative monads by construction, \Cref{eqformorphs} implies that $\iota$ is also a morphism of monoidal monads. 

\begin{cor}
 The inclusion $\iota:P\to V$ is a monoidal natural transformation. 
\end{cor}
In other words, the following diagram commutes, which implies that the product of two extendable continuous valuations is extendable. 
 $$
	\begin{tikzcd}
		PX \times PY \ar{r}{\nabla} \ar{d}{\iota\times \iota}	& P(X\times Y) \ar{d}{\iota} \\
		VX \times VY \ar{r}{\nabla}		& V(X\times Y)
	\end{tikzcd}
$$

\section{The support as a morphism of commutative monads}\label{secmonadmorph}

We now consider the support of a continuous valuation, and, more specifically, of a $\tau$-smooth Borel measure. Our main result is that taking supports is a natural transformation $\supp : V \Rightarrow H$, which is a morphism of commutative monads. Since $P \subseteq V$ is a commutative submonad, it follows that taking the supports of $\tau$-smooth probability measures is a morphism of commutative monads $\supp : P \Rightarrow H$.
The definition of support of a valuation is straightforward and analogous to the one of a measure, but it seems that it has not yet appeared in the literature, apart from a brief appearance in the concurrent work of Goubault-Larrecq and Jia~\cite[Example~4.4]{algebras}.

Intuitively, the support is the set of points of positive mass. The commonly used notion of support is the following, defined for a Borel measure $m$ on a space $X$. A measurable set $A \subseteq X$ has \emph{full measure} if and only if $m(X \setminus A)=0$.

\begin{deph}\label{defsupport}
	Let $X$ be a topological space and let $m$ be a Borel measure on $X$. The \emph{support} $\supp(m)$ of $m$ is the intersection of all closed subsets of $X$ that have full measure. 
\end{deph}

The support is the intersection of closed sets and therefore closed. The support of $\tau$-smooth measures is particularly well-behaved, as the following result shows.

\begin{prop}[Proposition~7.2.9 of~\cite{bogachev}]\label{tausmoothsupp}
Let $X$ be a topological space and let $m$ be a $\tau$-smooth Borel measure on $X$. Then $\supp(m)$ has full measure.
\end{prop}
\begin{proof}
The open set $U = X \setminus \supp (m)$ is the union of the family $(U_\lambda)_{\lambda\in\Lambda}$ of all open sets of measure zero. Given any two open sets of measure zero, their union has measure zero too, hence $(U_\lambda)_{\lambda\in\Lambda}$ is a directed net. The $\tau$-smoothness of $m$ implies
\begin{equation*}
m(U) = \sup_\lambda m(U_\lambda) = 0 .
\end{equation*}
Therefore $\supp(m) = X \setminus U$ has full measure. 
\end{proof}

It is clear from the proof that the support has full measure \emph{if and only if} $m$ is $\tau$-smooth on the open null sets. The following standard example, due to Dieudonné, shows that the support of a measure that is not $\tau$-smooth may not have full measure.

\begin{eg}[Dieudonné measure, Example~7.1.3 in~\cite{bogachev}]
\label{dieudonne}
We consider the initial segment of the ordinal numbers $X = [0, \omega_1]$, up to and including the first uncountable ordinal $\omega_1$. We equip this totally ordered set with the topology generated by intervals of the form $\{ x: a < x\}$, $\{ x: x < b\}$, or $\{ x: a < x < b\}$ for some $a,b \in X$. The Dieudonné measure is the Borel measure on $X$ defined as
\begin{equation*}
m (B) := \begin{cases}
	1 & \mbox{if there exists } F \subseteq B \setminus \{\omega_1\} \mbox{ which is closed and uncountable,} \\
	0 & \mbox{otherwise}.
\end{cases}
\end{equation*}
The measure $m$ is Borel~\cite[Examples~7.1.3 and~6.1.21]{bogachev}. The family $\left( [0,x) \right)_{x < \omega_1}$ is a directed net of open sets. We have $m \left( [0,x) \right) = 0$, since $[0,x)$ is countable for every $x < \omega_1$. But $\bigcup_{x < \omega_1} [0, x)$ is uncountable. We have
\begin{equation*}
0 = \sup_{x < \omega_1} m \left( [0,x) \right) < m \left( \bigcup_{x < \omega_1} [0,x) \right) = 1 .
\end{equation*}
Hence the Dieudonné measure is not $\tau$-smooth. Since every closed interval of the form $[x,\omega_1]$ for $x < \omega_1$ has full measure, we have $\supp (m) = \{ \omega_1\}$. But since $m( \{ \omega_1\} ) = 0$, the support does not have full measure.
\end{eg}

Due to the above example, some authors, for example Bogachev~\cite{bogachev}, require the support to have full measure by definition. In this case, some measures do not have a support. Since the present work focuses on $\tau$-smooth measures, the difference of definitions will not lead to ambiguity.

\begin{cor}\label{pointinsupport}
 Let $X$ be a topological  space, let $m$ be a $\tau$-smooth Borel measure on $X$, and let $U\subseteq X$ be open. Then $m(U)>0$ if and only if $U \cap \supp(m) \ne \varnothing$.  
\end{cor}

\subsection{The support transformation $V \Rightarrow H$}\label{precontsupport}

We turn to the case of valuations. In view of the discussion of the support of Borel measures, we give a definition of the support of a continuous valuation. It is phrased in the formalism of \Cref{hyperspace}, and uses the function $\sgn : \Rplusext \to S$ defined as
\[
	\sgn(r) \;\coloneqq\;
	\begin{cases}
		1 & \text{if } r > 0, \\
		0 & \text{otherwise},
	\end{cases}
\]
and the function $\ngs : S \to \Rplusext$ defined as $\ngs(0) \coloneqq 0$ and $\ngs(1) \coloneqq \infty$.

Note that both of these maps are continuous, and that $\ngs$ is right adjoint to $\sgn$. We therefore have a universal way of turning an open $U \in S^X$ into a function $\ngs \circ U \in \Rplusext^X$, namely the function that is infinite on $U$ and vanishes elsewhere.

\begin{deph}[Support of a continuous valuation]\label{newdefsupport}
	Let $X$ be a topological space and consider $\nu \in VX$. The support of $\nu$ is the closed set $\supp(\nu) \in HX$ characterized by
	\begin{equation}
		\label{supp_def}
		\pair{\supp(\nu)}{U} \;=\; \sgn( \pair{\nu}{\ngs \circ U} )
	\end{equation}
	for every open $U \in S^X$.
\end{deph}

Thus $\supp(\nu)$ is defined such that pairing it with an open set consists of forcing the open set to take values in the right codomain in a universal way, applying the pairing with $\nu$, and then transporting back in a universal way. As with the double dualization construction of $H$ and $V$, the abstract definition~\eqref{supp_def} facilitates proofs by making them into mere unfoldings of definitions.

Phrased more concretely,~\eqref{supp_def} states that an open set $U$ is disjoint from $\supp(\nu)$ if and only if $\pair{\nu}{\ngs \circ U} = 0$. The definition of $\ngs$ and the linearity of $\nu$ imply that this happens if and only if $\nu(U) = 0$. In other words, the support is characterized as the largest closed set with the property $\nu(X \setminus \supp(\nu)) = 0$. In particular, the support of an extendable valuation equals the support of its extension to a Borel measure as defined in \Cref{defsupport}.

We want to record the following intuitive statement, which will not play any further role in this article.

\begin{prop}
Let $X$ be a topological space, $\nu \in VX$ a continuous valuation, and $f \in \Rplusext^X$. Then
\begin{equation}\label{signsupp}
	\sgn(\pair{\nu}{f}) = \pair{\supp(\nu)}{\sgn \circ f} .
\end{equation}
\end{prop}

In down-to-earth terms, this means that $\pair{\nu}{f} = 0$ is equivalent to $f|_{\supp(\nu)} = 0$, which is a property that one would expect from a well-behaved notion of support.

\begin{proof}
	By \Cref{supp_def}, we have to show that
	\[
		\sgn(\pair{\nu}{f}) \;=\; \sgn(\pair{\nu}{\ngs \circ \sgn \circ f}),
	\]
	which means that the integral of $f$ vanishes if and only if the integral of the function where all positive values are rounded up to $\infty$ vanishes. Since, with respect to scalar multiplication from $\Rplusext$, we have $\ngs \circ \sgn \circ f = \infty f$ this is a consequence of the $\Rplusext$-linearity of the pairing
$$
\sgn(\pair{\nu}{\ngs \circ \sgn \circ f}) \;=\; \sgn\left( \pair{\nu}{\infty f} \right) \;=\; \sgn\left( \infty \, \pair{\nu}{n\, f} \right) \;=\; \sgn(\pair{\nu}{f}) ,
$$
where the last step can be thought of as using the equation $\sgn \circ \ngs \circ \sgn = \sgn$.
\end{proof}

We now aim towards showing that $\supp : V \Rightarrow H$ is a morphism of monads. Its continuity is obvious as the right-hand side of~\eqref{supp_def} is continuous in $\nu$. Similar continuity statements for different choices of topologies are known (such as Theorem~17.1 in \cite{infdynana}). We want to stress that the continuity of the support map for the lower Vietoris topology on $HX$ can be interpreted as \emph{lower semicontiuity} rather than continuity, as the following example shows. 

\begin{eg}\label{suppislsc}
Let $X$ be the two-element set $\{x,y\}$ equipped with the discrete topology. Consider the following sequence of measures (or the corresponding valuations),
$$
\nu_n = \dfrac{n-1}{n} \,\delta_x + \dfrac{1}{n} \,\delta_y
$$
for $n\in \N$. In $VX$ (or in $PX$) this sequence tends to $\delta_x$. We have $\supp (\nu_n) = \{x,y\}$ for all $n \in N$, while $\supp(\delta_x)=x$. Hence
$$
 \supp \left( \lim_{n\to\infty} \nu_n \right) \subsetneq \lim_{n\to\infty} \supp(\nu_n) .
$$
The support of the limit can be smaller than the limit of the supports, a property that parallels the lower semicontinuity of real functions.
\end{eg}

\noindent We now turn to naturality, which relies on the Scott continuity of valuations.

\begin{prop}\label{propnaturality}
 Let $f: X \to Y$ be continuous. The following diagram commutes.
 \begin{equation*}
 \begin{tikzcd}
  VX \ar{r}{\supp} \ar{d}{f_*} & HX \ar{d}{f_\sharp} \\
  VY \ar{r}{\supp} & HY
 \end{tikzcd}
 \end{equation*}
\end{prop}
\begin{proof}  
	For every $U \in S^Y$, we have
	\begin{align*}
		\pair{\supp(f_*\nu)}{U} \;&=\; \sgn(\pair{f_*\nu}{\ngs \circ U}) \;=\; \sgn(\pair{\nu}{\ngs \circ U \circ f}) \\
		&=\; \pair{\supp(\nu)}{U \circ f} \;=\; \pair{f_\sharp(\supp(\nu))}{U} . \qedhere
	\end{align*}
\end{proof}

Therefore the support induces a natural transformation $\supp:V \Rightarrow H$, which can be restricted to a natural transformation $\supp : P \Rightarrow H$. On the set of \emph{all} Borel measures, the support as given in \Cref{defsupport} would not be natural. We give an example using the Dieudonné measure of \Cref{dieudonne}.

\begin{eg}[Dieudonné measure, continued]
Consider $Y=[0, \omega_1)$, the space of countable ordinals (which can be identified with $\omega_1$). We equip this totally ordered set with the subspace topology of $X$ from \Cref{dieudonne} and consider the Dieudonné measure $m$ from \Cref{dieudonne}, which is not $\tau$-smooth. The measure $m$ restricted to $Y$ has $\supp(m) = \varnothing$. Consider the map $f: Y \to \{\bullet\}$ into the singleton space. We have $\supp ( f_* m ) = \supp ( \delta_\bullet ) = \{ \bullet \}$ while $f_\# \supp (m) = f_\# \varnothing = \varnothing$. 
\end{eg}

The following counterexample shows that the notion of support is also not natural for signed measures, which is why we do not expect a generalization of our results to the signed case.

\begin{eg}
	Consider the two-point space $\{a,b\}$ with the finite signed measure $m := \delta_a - \delta_b$. The map into the singleton space $f: \{a,b\} \to \{\bullet\}$ yields $f_* m = 0$. The usual definition of the support of a signed measure leads to $\supp (m) = \{a,b\}$, and hence $f_\sharp(\supp(m)) = 1$. But we have $\supp \left( f_* m \right) = \varnothing$.
\end{eg}

\subsection{The support is a morphism of monads}

\begin{prop}\label{propmmorph}
 For any topological space $X$, the following two diagrams commute.
 \[
  \begin{tikzcd}
    & VX \ar{dd}{\supp} \\
   X \ar{ur}{\delta} \ar[swap]{dr}{\sigma} \\
    & HX 
  \end{tikzcd}
  \qquad\qquad
  \begin{tikzcd}
    VVX \ar{d}{\supp} \ar{r}{\E} & VX \ar{dd}{\supp} \\
    HVX \ar{d}{\supp_\sharp} \\
    HHX \ar{r}{\U} & HX 
  \end{tikzcd}
 \]
\end{prop}

We refer to these diagrams as the \emph{unit} and the \emph{multiplication} diagram, since they express the compatibility of the support map with the respective monad maps. The unit diagram says that the support of the Dirac valuation $\delta(x)$ equals $\cl(\{x\})$. The multiplication diagram says: If the valuation $\nu\in VX$ is the integral of the valuation $\xi \in VVX$, then the support of $\nu$ is the closure of the union of the supports of all the valuations in the support of $\xi$.

To illustrate this statement in the simplest case, suppose that $X$ is finite and discrete and that $\xi \in VVX$ is finitely supported. Then $VX$ can be identified with the simplex of probability vectors $(p_x)_{x\in X}$ and $\E\xi$ is a finite convex combination of these. The statement is then that the set of nonzero components of $\E\xi$ coincides with the union of the sets of nonzero components in each term contributing to the convex combination. In the discrete case this statement is straightforward to prove, however, to the best of our knowledge, it has never appeared in a published document.\footnote{We did receive an independent proof of the statement for the finite case from G.~van Heerdt, J.~Hsu, J.~Ouaknine, and A.~Silva in a personal communication.} Our result can be thought of as a generalization to continuous distributions. Despite this, the proof is purely formal and does not require any topological or measure-theoretic input.

\begin{proof}[Proof of \Cref{propmmorph}] 
Considering the unit diagram, we have, for any $x\in X$ and any open $U \in S^X$,
$$
\pair{\supp(\delta(x))}{U} \;=\; \sgn(\pair{\delta(x)}{\ngs \circ U}) \;=\; \sgn( \ngs(U(x)) \;=\; U(x) \;=\; \pair{\sigma(x)}{U} .
$$
Considering the multiplication diagram, let $\xi \in VVX$. Using \Cref{signsupp} in the third and fourth step, we obtain that, for any $U \in S^X$,
	\begin{align*}
		\pair{\supp(\E\xi)}{U} \;&=\; \sgn(\pair{\E\xi}{\ngs \circ U}) \;=\; \sgn(\pair{\xi}{\pair{-}{\ngs \circ U}}) \\
		&=\; \pair{\supp(\xi)}{\sgn(\pair{-}{\ngs \circ U})} \;=\; \pair{\supp(\xi)}{\pair{\supp(-)}{U}} \\
		&=\; \pair{\supp_\sharp(\supp(\xi))}{\pair{-}{U}} \;=\; \pair{\U(\supp_\sharp(\supp(\xi)))}{U}. \qedhere
	\end{align*}
\end{proof}

\noindent We can summarize our results so far as follows.

\begin{cor}\label{thmmonadmorph}
 The formation of the support of a continuous valuation is a morphism of monads
 \[
	 \supp:(V,\delta,\E)\longrightarrow (H,\sigma,\U).
 \]
\end{cor}

\subsection{Consequences for algebras}
\label{conseqalg}

It is well-known that a morphism of monads on the same category induces a functor between the respective categories of algebras in the opposite direction.

\begin{prop}
 Let $(S,\eta_S,\mu_S)$ and $(T,\eta_T,\mu_T)$ be monads on a category $\cat{C}$ and let $m:S\Rightarrow T$ be a morphism of monads. Then every $T$-algebra $(A,a)$ can be equipped with an $S$-algebra  structure via $ (A,a) \mapsto (A, a\circ m)$. Moreover, a $T$-algebra morphism $f:(A,a)\to(B,b)$ induces an $S$-algebra morphism $f: (A, a \circ m) \to (B, b \circ m)$ in a functorial way.
\end{prop}

\noindent The above combined with \Cref{thmmonadmorph} yields the following.

\begin{cor}
	Every $H$-algebra is a $V$-algebra, and therefore also a $P$-algebra, in a canonical way. Concretely, if $(A,a)$ is an $H$-algebra, then $(A, a \circ \supp)$ is a $V$-algebra.
\end{cor}

The algebras of $H$ are the complete topological join-semilattices (\Cref{Halg}). Hence we obtain the following statement.

\begin{cor}
 Every topological complete join-semilattice is a $V$-algebra with structure map $\nu \mapsto \bigjoin \supp(\nu)$.
\end{cor}

A very similar result was obtained independently by Goubault-Larrecq and Jia \cite[Proposition 4.14]{algebras}.

Since every $V$-algebra is also a topological cone (\Cref{Valg}), it follows that every topological complete join-semilattice $A$ is a topological cone. Unfolding the definitions shows that its addition is given by the binary join, and its scalar multiplication is
 $$
 r x  \;=\; \begin{cases}
	 x &\text{if $r > 0$},\\
		 \bot & \text{if $r = 0$},
                               \end{cases}
 $$
 which is indeed jointly continuous. This is the same in spirit as the convex spaces of combinatorial type \cite{fritz}.

\subsection{The support of products and marginals}

The support is not only a morphism of monads, but also respects the strengths (or equivalently the monoidal structures).

\begin{prop}
 Let $X$ and $Y$ be topological spaces. The following diagram commutes. 
 $$
 \begin{tikzcd}
  X \times VY \ar{d}[swap]{\id\times\supp} \ar{r}{s} & V(X\times Y) \ar{d}{\supp} \\
  X\times HY \ar{r}{s} & H(X\times Y)
 \end{tikzcd}
 $$
\end{prop}
\begin{proof}
	Consider $x\in X$ and $\rho\in VY$. Since the sign function is multiplicative, we have, for every $U\subseteq X$ and $V\subseteq Y$,
	\begin{align*}
		\pair{\supp(s(x,\rho))}{U\times V} \;&=\; \sgn(\pair{s(x,\rho)}{\ngs \circ (U\times V)} \\
		&=\; \sgn(\pair{s(x,\rho)}{(\ngs \circ U) \times (\ngs \circ V)}) \\
		&=\; \sgn\left( \ngs(U(x)) \, \pair{\rho}{\ngs \circ V} \right) \\
		&=\; \sgn(\ngs(U(x))) \, \sgn(\pair{\rho}{\1_V}) \\
		&=\; U(x) \, \pair{\supp(\rho)}{V} \\
		&=\; \pair{s(x,\supp(\rho))}{U\times V} .
	\end{align*}
This is enough by \Cref{H_basis}.
\end{proof}

\begin{cor}
 The support map $\supp : V \to H$ is a monoidal natural transformation. 
\end{cor}
\begin{proof}
See \Cref{eqformorphs}.
\end{proof}

\noindent In particular, the following diagram commutes. 
 $$
 \begin{tikzcd}
  VX \times VY \ar{d}[swap]{\supp\times\supp} \ar{r}{\nabla} & V(X\times Y) \ar{d}{\supp} \\
  HX\times HY \ar{r}{\nabla} & H(X\times Y)
 \end{tikzcd}
 $$
In other words, the support of the product distribution is the product of the supports of the factors. For the unitors, we have the trivial statement that the support of the valuation in the image of $u:1\to V1$ is the unique point of $1$.

By the universal property of the Cartesian product, we know that the support is also an opmonoidal natural transformation. This means that the supports of the marginals are the projections of the support. 
Indeed, the product projection $\pi_1:X\times Y\to X$ makes the following diagram commute,
$$
\begin{tikzcd}[column sep=small, row sep=small]
 & V(X\times Y) \ar{dd}[near end]{(\pi_1)_*} \ar{dl} \ar{rrr}{\supp} &&& H(X\times Y) \ar{dl} \ar{dd}{(\pi_1)_\sharp} \\
 VX \times VY \ar{dr}[swap]{\pi_1} \ar[crossing over]{rrr}[near end]{\supp\times\supp} &&& HX \times HY \ar{dr}{\pi_1} \\
 & VX \ar{rrr}[swap]{\supp} &&& HX
\end{tikzcd}
$$
and the same can be said about $\pi_2:X\times Y\to Y$. The bottom parallelogram of this diagram gives a map $VX\times VY\to HX$, and similarly we obtain a map $VX\times VY\to HY$. By the universal property of the product, there is an induced map $VX\times VY\to HX\times HY$ commuting with the supports, and this is the opmonoidal structure. The commutativity of the square face of the diagram implies that the supports of the marginals are the projections of the support.

\appendix

\section{Top as a 2-category}
\label{2cat}

Here we recall some background material on the specialization preorder, which makes $\cat{Top}$ into a category enriched in preordered sets, and therefore into a 2-category. Every topological space is canonically equipped with a preorder, its specialization preorder. Since a preorder is a category, this equips $\cat{Top}$ with a 2-categorical structure. This is not the higher categorical structure that arises from homotopies of maps. Rather, it is closely related to the usual 2-categorical structure on the category of locales.

\begin{deph}[Specialization preorder]
	Let $X$ be a topological space. Given $x,y\in X$, we define the specialization preorder on $X$ by setting $x\le y$ if and only if
	\[
		x \in U \quad\Longrightarrow\quad y \in U
	\]
	for every $U \in \op(X)$.
\end{deph}

Equivalently, $x\le y$ if and only if $x \in \cl(\{y\})$, or if and only if $\cl(\{x\}) \subseteq \cl(\{y\})$. Yet equivalently, any net or filter which converges to $y$ also converges to $x$.

It is obvious from the definition that the specialization preorder is indeed a preorder (it is reflexive and transitive). A space is $T_0$ if and only if its specialization preorder is a partial order (it is antisymmetric). A space is $T_1$ if and only if its specialization preorder is the discrete relation. Every preorder on any set arises as the specialization preorder of some topology, for example its Alexandrov topology.

We write $x \sim y$ if $x \leq y$ as well as $y \leq x$. Any two equivalent points $x \sim y$ have the same neighborhood filter, and any net or filter tending to one also tends to the other. Since two points of a space $X$ are equivalent in the specialization preorder if and only if they have the same open neighborhoods, a space is $T_0$ if and only if $x \sim y$ implies $x=y$. The \emph{Kolmogorov quotient} is the quotient space $X/\sim$, which is the initial $T_0$ space equipped with a continuous map from $X$. Hence the category of $T_0$ spaces is a reflective subcategory of $\cat{Top}$.

\begin{deph}
	\label{2cell_def}
	Let $f,g: X \to Y$ be continuous maps. We say that there is a 2-cell $f\le g$ if and only if, for every $x\in X$, we have $f(x)\le g(x)$ in the specialization preorder of $Y$.
\end{deph}

Every continuous map is monotone for the specialization preorder. Hence, equipped with these 2-cells, $\cat{Top}$ is a strict 2-category.

\begin{lemma}\label{ordopen}
 Let $f,g:X\to Y$ be continuous maps. Then $f\le g$ if and only if $f^{-1}(U)\subseteq g^{-1}(U)$ for every open set $U\subseteq Y$.
\end{lemma}
\begin{proof}
Assuming $f\le g$, we prove the claim by showing that every $x\in f^{-1}(U)$ is also in $g^{-1}(U)$. Since $x\in U$ and $f(x) \le g(x)$, we indeed have $g(x) \in U$. 

Suppose that $f^{-1}(U)\subseteq g^{-1}(U)$ for all open sets $U \subseteq Y$. Then, for every $x\in X$ and every open $U \ni f(x)$, we have $x\in f^{-1}(U)\subseteq g^{-1}(U)$, which implies that $g(x) \in U$. We have shown that every open neighborhood of $f(x)$ is an open neighborhood of $g(x)$. 
\end{proof}

\section{Topology of mapping spaces}
\label{functop}

Here we briefly discuss the issue of equipping the set of maps between topological spaces with a suitable topology. This is necessary to understand the precise sense in which our monads $H$ and $V$ from \Cref{hyperspace,sec_valuations} are double dualization monads, and in particular to answer whether these monads are covered by the general framework of Lucyshyn-Wright for double dualization monads~\cite{lucyshyn-wright}. This framework is designed to apply to Cartesian closed categories, which $\cat{Top}$ is well-known not to be. Nevertheless, since we only dualize with respect to particular spaces $S$ and $\Rplusext$, one may hope that it could be enough if exponential objects of the form $S^X$ and $\Rplusext^X$ existed in $\cat{Top}$.

As we will recall here, this is not the case. Thus, although our monads unmistakably have the flavor of double dualization monads, we are not aware of any existing categorical framework for double dualization monads which would accommodate them, and the development of such a framework seems like an interesting problem for future work.

To begin the discussion of exponential objects in $\cat{Top}$, recall that a subset $V \subseteq U$ of a topological space is \emph{relatively compact} in $U$ if and only if every open cover of $U$ admits a finite subcover of $V$, or, equivalently, if and only if $V$ is way below $U$ with respect to the inclusion order~\cite[p.~50]{continuous}. For this reason, we also denote relative compactness by $V \ll U$.

\begin{deph}[Core-compact spaces~\cite{HL}]
 A topological space $X$ is core-compact if, for every $x \in X$ and every open neighborhood $U \ni x$, there exists an open neighborhood $V$ such that $x \in V \subseteq U$ and $V \ll U$.
\end{deph}

It follows that a space is core-compact if and only if its lattice of open sets is a continuous lattice~\cite{isbell}. If $X$ is sober, then it is core-compact if and only if it is locally compact (see Theorem~V-5.6 in~\cite{continuous}).

\begin{thm}[See \cite{isbell,EH} and Proposition II-4.6 in \cite{continuous}]
	Let $X$ be a topological space. The functor $- \times X : \cat{Top} \to \cat{Top}$ has a right adjoint $(-)^X : \cat{Top} \to \cat{Top}$ if and only if $X$ is core-compact.
\end{thm}

Whenever $X$ is not core-compact, the exponential space $S^X$ does not exist, and neither does $\Rplusext^X$. This is well-known and may be shown by translating a theorem of Niefield \cite[Theorem 2.3]{niefield} into our setting. (Niefield's ``cartesianness'' corresponds to our ``exponentiability'' and we take her $T$ to be $1$ so that $\cat{Top}/T$ becomes $\cat{Top}$.) The equivalence of conditions (b), (c) and (d) in Niefield's theorem then reads as follows.

\begin{thm}[Niefield]
\label{niefield}
Let $X$ be a topological space. The following three conditions are equivalent.
 \begin{enumerate}
  \item $X$ is exponentiable.
  \item The exponential object $S^X$ exists in $\cat{Top}$.
  \item\label{filtercond} Given any open set $U\subseteq X$ and any point $x\in U$, there exists a Scott-open $\mathcal{V}\subseteq \op(X)$ such that $U\in\mathcal{V}$ and $\bigcap\mathcal{V}$ is a neighborhood of $x$.
 \end{enumerate}
\end{thm}

\begin{prop}\label{corecompacteq}
 Condition \ref{filtercond} holds if and only if $X$ is core-compact.
\end{prop}

\begin{proof}[Proof of \ref{corecompacteq}]
	Suppose that $X$ satisfies \ref{filtercond}. Then, given $x \in X$ and an open neighborhood $U \ni x$, we have a Scott-open $\mathcal{V} \subseteq \op(X)$ such that $U\in \mathcal{V}$ and $\bigcap \mathcal{V}$ is a neighborhood of $x$. The latter means that there exists an open set $V$ such that $x\in V \subseteq \bigcap \mathcal{V}$. We claim that $V\ll U$. Let $\{U_\alpha\}_{\alpha\in A}$ be an open cover of $U$. Since $U\in\mathcal{V}$, Scott-openness implies that there is a finite subfamily $\left( U_{\alpha_i} \right)_{i=1}^n$ such that $\bigcup_{i=1}^n U_{\alpha_i}$ is a member of $\mathcal{V}$, and therefore contains $\bigcap \mathcal{V}$. But then this finite subfamily also covers $V$.
 
 Conversely, let $X$ be core-compact. For an open set $U\subseteq X$ and $x\in U$, there exists an open $V \subseteq X$ such that $x\in V \ll U$. Define the set 
 $$
 \mathcal{V} \; \coloneqq \; \{U'\in \op (X) : V \ll U' \} .
 $$
 We have $U\in \mathcal{V}$ and $V\subseteq \bigcap\mathcal{V}$. The set $\mathcal{V} \subseteq \op(X)$ is Scott-open by the assumption that $\op(X)$ is a continuous lattice, which implies that principal upsets with respect to $\ll$ are Scott-open (Proposition~II.1.6 in~\cite{continuous}).
\end{proof}

\section{Commutative and symmetric monoidal monads}
\label{comm_monads}

We recall the notion of monad and the equivalence between a commutative monad and a symmetric monoidal monad. We assume familiarity with symmetric monoidal categories, lax symmetric monoidal functors, and monoidal natural transformations. All our monoidal functors will be lax monoidal. The monoidal category of primary interest in the main text is $\cat{Top}$ with its Cartesian product structure. We start with the definition of monads and their morphisms for convenient reference.

\begin{deph}[Category of monads]
	Let $\cat{C}$ be a category.
	\begin{enumerate}
		\item A functor $T : \cat{C} \to \cat{C}$ is a monad if it is equipped with natural transformations $\eta : \id \Rightarrow T$ and $\mu : TT \Rightarrow T$ such that the following three diagrams commute for every $X \in \cat{C}$.
 \begin{equation}
  \begin{tikzcd}
   TX \ar{r}{\eta} \idar{dr} & TTX \ar{d}{\mu}  \\
   & TX
  \end{tikzcd}
  \qquad
  \begin{tikzcd}
   TX \ar{r}{T\eta} \idar{dr} & TTX \ar{d}{\mu} \\
   & TX
  \end{tikzcd}
  \qquad
  \begin{tikzcd}
   TTTX \ar{r}{T\mu} \ar{d}{\mu} & TTX \ar{d}{\mu} \\
   TTX \ar{r}{\mu} & TX
  \end{tikzcd}
 \end{equation}
 		\item If $(T,\mu,\nu)$ and $(T',\mu',\nu')$ are monads, then a morphism of monads is a natural transformation $\alpha : T \to T'$ such that the following two diagrams commute for every $X \in \cat{C}$.
 \begin{equation}
  \begin{tikzcd}
    & TX \ar{dd}{\alpha} \\
   X \ar{ur}{\eta'} \ar[swap]{dr}{\eta} \\
    & T'X 
  \end{tikzcd}
  \qquad\qquad
  \begin{tikzcd}
    TTX \ar{d}{\alpha} \ar{r}{\mu} & TX \ar{dd}{\alpha} \\
    T'TX \ar{d}{\alpha} \\
    T'T'X \ar{r}{\mu'} & T'X 
  \end{tikzcd}
 \end{equation}
 \end{enumerate}
\end{deph}

An introduction to monads from the probabilistic perspective is given by Perrone \cite[Chapter~1]{thesis}. We use the term \emph{probability monad} to loosely refer to monads $T$ where for every object $X \in \cat{C}$, the object $TX$ is the object of probability measures (of a given kind) on $X$. Then $T(X\otimes Y)$ is the object of \emph{joint} probability distributions on $X$ and $Y$, where the tensor product $\otimes$ is often the Cartesian product, but not always~\cite{ours_bimonoidal}. The multiplication $\mu : TTX \to TX$ averages a probability measure on probability measures to the probability measure representing the expectation value or barycenter of the measure on measures. All three monads considered in this paper are variations on this theme.

An important structure in probability theory is the formation of product distributions. In the probability monad formalism, this is encoded as a natural transformation $\nabla : TX \otimes TY \to T(X \otimes Y)$ which makes $T$ into a lax symmetric monoidal functor that interacts nicely with the monad structure.

\begin{deph}[Symmetric monoidal monad]
    Let $\cat{C}$ be a symmetric monoidal category.
	Suppose that a functor $T : \cat{C} \to \cat{C}$ carries both the structure of a monad and of a symmetric monoidal functor with structure maps $\nabla : T(-) \otimes T(-) \to T(- \otimes -)$ and $u : 1 \to T1$. Then $T$ is a symmetric monoidal monad if $u = \eta$ as morphisms $1 \to T1$ and the following two diagrams commute.
	\[
		\begin{tikzcd}
			& X \times Y \ar{ld}[swap]{\eta \otimes \eta} \ar{rd}{\eta} \\
			 TX \times TY \ar{rr}{\nabla} & & T(X\times Y)
		\end{tikzcd}
	\]\medskip
	\[
	\begin{tikzcd}
			 TTX \times TTY \ar{d}{\mu\times\mu} \ar{r}{\nabla} & T(TX\times TY) \ar{r}{T\nabla} & TT(X\times Y) \ar{d}{\mu} \\
			 TX\times TX \ar{rr}{\nabla} && T(X\times Y) 
		\end{tikzcd}
	\]
\end{deph}

It is well-known that, given a monad $T$, there is a bijective correspondence between symmetric monoidal structures on $T$ and \emph{strengths} on $T$ which make $T$ into a \emph{commutative monad}~\cite{kock}. We recall the relevant definitions, still on a symmetric monoidal category $\cat{C}$.

\begin{deph}[Strength]
    Let $\cat{C}$ be a monoidal category and $T$ a monad on $\cat{C}$.
	A strength on a monad $T$ is a family of maps $X \otimes TY \to T(X \otimes Y)$, natural in $X$ and $Y$, such that the following four diagrams commute for all $X,Y\in\cat{C}$, where the unnamed isomorphisms are the monoidal structure isomorphisms.
	\[
		\begin{tikzcd}
			1 \otimes TX \ar{r}{s} \ar{dr}[swap]{\cong} & T(1\otimes X) \ar{d}{\cong} \\
			 & TX
		\end{tikzcd}
	\]\medskip
	\[
		\begin{tikzcd}
			 (X\otimes Y) \otimes TZ \ar{rr}{s} \ar{d}{\cong} && T((X\otimes Y)\otimes Z) \ar{d}{\cong} \\
			 X\otimes(Y\otimes TZ) \ar{r}{\id\otimes s} & X\otimes T(Y\otimes Z) \ar{r}{s} & T(X\otimes(Y\otimes Z))
		\end{tikzcd}
	\]\medskip
	\[
		\begin{tikzcd}
			 X\otimes Y \ar{r}{\id\otimes\eta} \ar{dr}[swap]{\eta} & X\otimes TY \ar{d}{s} \\
			 & T(X\otimes Y)
		\end{tikzcd}
	\]\medskip
	\[
		\begin{tikzcd}
			 X\otimes TTY \ar{r}{s} \ar{d}{\id\otimes\mu} & T(X\otimes TY) \ar{r}{Ts} & TT(X\otimes Y) \ar{d}{\mu} \\
			 X\otimes TY \ar{rr}{s} && T(X\otimes Y)
		\end{tikzcd}
	\]
\end{deph}

Similarly, a \emph{costrength} is a natural transformation with components $t : TX \otimes Y \to T(X \otimes Y)$ satisfying the analogous equations. Since $\cat{C}$ is symmetric monoidal, every strength induces a costrength and vice versa.
A \emph{strong monad} is a monad equipped with a strength.

\begin{deph}[Commutative monad]
    	Let $\cat{C}$ be a symmetric monoidal category, and
	$T$ a strong monad on $\cat{C}$ with strength $s$ and the costrength $t$ induced from the braiding. Then $T$ is a commutative monad if the following diagram commutes for all $X,Y \in \cat{C}$.
	\begin{equation}
		\label{commutative_diagram}
 		\begin{tikzcd}
			TX \otimes TY \ar{d}{t} \ar{r}{s} & T(TX\otimes Y) \ar{r}{Tt} & TT(X\otimes Y) \ar{d}{\mu} \\
			T(X\otimes TY) \ar{r}{Ts} & TT(X\otimes Y) \ar{r}{\mu} & T(X\otimes Y)
		\end{tikzcd}
	\end{equation}
\end{deph}

The bijection between a symmetric monoidal structure on a monad $T$ and a commutative structure on $T$ can be obtained by explicit construction of each piece of structure in terms of the other. In one direction, we start with $\nabla : TX \otimes TY \to T(X \otimes Y)$ and obtain a strength as the composite
\[
	\begin{tikzcd}
		X \otimes TY \ar{r}{\eta\otimes\id} & TX \otimes TY \ar{r}{\nabla} & T(X \otimes Y).
	\end{tikzcd}
\]
In the other direction, we start with $s : X \otimes TY \to T(X \otimes Y)$ and obtain a symmetric monoidal structure as the diagonal of~\Cref{commutative_diagram}.

The following observation is fairly elementary, but seems to be hard to find in the literature. It says that the correspondence between commutative strong monads and symmetric monoidal monads is functorial.

\begin{prop}\label{eqformorphs}
	If $S$ and $T$ are commutative monads, then the following two properties of a morphism of monads $\alpha : S \Rightarrow T$ are equivalent.
	\begin{enumerate}
		\item\label{monoidal_trafo}The morphism $\alpha$ is a monoidal natural transformation, that is the following two diagrams commute for all $X,Y\in\cat{C}$.
		\begin{equation}\label{monoidalmorph}
		 \begin{tikzcd}[row sep=small]
		  & S1 \ar{dd}{\alpha} \\
		  1 \ar{ur}{\eta} \ar[swap]{dr}{\eta} \\
		  & T1
		 \end{tikzcd}
		 \qquad
		 \begin{tikzcd}[row sep=small]
		  SX\otimes SY \ar{dd}{\alpha\otimes\alpha} \ar{r}{\nabla} & S(X\otimes Y) \ar{dd}{\alpha} \\ \\
		  TX\otimes TX \ar{r}{\nabla} & T(X\otimes Y)
		 \end{tikzcd}
		\end{equation}.
		\item\label{strong_trafo}The morphism $\alpha$ preserves the strengths in the sense that the following diagram commutes for all $X,Y \in \cat{C}$.
			\begin{equation}\label{strongmorph}
				\begin{tikzcd}
					X \otimes S Y \ar{r}{s} \ar{d}{\id\otimes\alpha}	& S (X \otimes Y) \ar{d}{\alpha}	\\
					X \otimes TY \ar{r}{s'}				& T(X \otimes Y)
				\end{tikzcd}
			\end{equation}
	\end{enumerate}
\end{prop}

The proof proceeds by construction of the lax monoidal structure from the strength and vice versa. Both directions use the assumption that $\alpha$ preserves the monad structure. The more tedious direction from \ref{strong_trafo} to \ref{monoidal_trafo} also uses the naturality of $s$ and $\alpha$. Since the proof is hard to locate in the literature, we give it here.

\begin{proof}
 We start by \ref{monoidal_trafo}$\Rightarrow$\ref{strong_trafo}. We can decompose the diagram \Cref{strongmorph} as follows.
 $$
 \begin{tikzcd}
  X\otimes SY \ar[swap]{dr}{\eta\otimes \id} \ar{rr}{s} \ar{ddd}[swap]{\id\otimes\alpha} && S(X\otimes Y) \ar{ddd}{\alpha} \\
  & SX\otimes SY \ar{ur}[swap]{\nabla} \ar{d}{\alpha\otimes\alpha} \\
  & TX\otimes TX \ar{dr}{\nabla} \\
  X\otimes TY \ar{ur}{\eta\otimes\id} \ar{rr}{s} && T(X\otimes Y)
 \end{tikzcd}
 $$
 The triangles at the top and at the bottom commute; in fact they are the standard way of obtaining a strength from a monoidal structure.  The trapezium on the left commutes because $\alpha$ is a morphism of monads and therefore preserves the units. The trapezium on the right is the second diagram of \Cref{monoidalmorph}. 
 
We turn to \ref{strong_trafo}$\Rightarrow$\ref{monoidal_trafo}. The first diagram of \Cref{monoidalmorph} commutes since $\alpha$ is a morphism of monads and therefore preserves the units. The second diagram of \Cref{monoidalmorph} can be decomposed as follows. 
 $$
 \begin{tikzcd}
  SX\otimes SY \ar{dr}[swap]{s} \ar{dd}[swap]{\id\otimes\alpha} \ar{rrr}{\nabla} &&& S(X\otimes Y) \ar{dddd}{\alpha} \\
  & S(SX\otimes Y) \ar{r}{St} \ar{d}{\alpha} & SS(X\otimes Y) \ar{ur}[swap]{\mu} \ar{d}{\alpha} \\
  SX\otimes TY \ar{r}{s} \ar{dd}[swap]{\alpha\otimes\id} & T(SX\otimes Y) \ar{r}{Tt} \ar{d}{T(\alpha\otimes\id)} & TS(X\otimes Y) \ar{d}{T\alpha} \\
  & T(TX\otimes Y) \ar{r}{Tt} & TT(X\otimes Y) \ar{dr}{\mu} \\
  TX\otimes TY \ar{ur}{s} \ar{rrr}{\nabla} &&& T(X\otimes Y)
 \end{tikzcd}
 $$
 The upper and lower trapezia commute; in fact they are the standard way of obtaining a monoidal structure from a strength. The upper trapezium on the left is the diagram \Cref{strongmorph}, which commutes by hypothesis. The lower trapezium on the left commutes by naturality of $s$. The upper central square commutes by naturality of $\alpha$. The lower central square is the image under $T$ of a naturality square for $t$. The trapezium on the right commutes since $\alpha$ is a morphism of monads and therefore preserves the multiplication.
\end{proof}

In probability theory, one is interested in the formation of \emph{marginals}, which is implemented by an \emph{opmonoidal} structure $T(X \otimes Y) \to TX \otimes TY$~\cite{ours_bimonoidal}. In the setting of this paper, where $\cat{C} = \cat{Top}$, the monoidal structure of our category is the Cartesian product structure. Hence the projections $X \times Y \to X$ and $X \times Y \to Y$ induce maps $T(X \times Y) \to TX$ and $T(X \times Y) \to TY$ by functoriality. These induce the opmonoidal structure describing the formation of marginal distributions from joint distributions. The well-behaved interaction of this opmonoidal structure with the monoidal structure is immediately implied \cite[Proposition~3.4]{ours_bimonoidal}.

\printbibliography
\addcontentsline{toc}{section}{\bibname}

\end{document}